\documentclass[12pt]{amsart} 
\usepackage{amscd,amsthm,amssymb,amsfonts,amsmath,epsfig,epic,bbm,url, latexsym}
\usepackage[arrow,matrix,tips,2cell]{xy}

\theoremstyle{plain}
\newtheorem{theorem}{Theorem}

\newtheorem{lemma}[theorem]{Lemma}
\newtheorem{proposition}[theorem]{Proposition}
\newtheorem{prop}[theorem]{Proposition}
\newtheorem{corollary}[theorem]{Corollary}

\theoremstyle{definition}
\newtheorem{definition}[theorem]{Definition}
\newtheorem{defn}[theorem]{Definition}
\newtheorem{example}[theorem]{Example}

\theoremstyle{remark}
\newtheorem*{remark}{Remark}
\newtheorem*{rmk}{Remark}
\newtheorem*{acknowledgments}{Acknowledgments}

\newcommand{\Z}{\mathbb Z}    
\newcommand{\Q}{\mathbb Q}    
\newcommand{\R}{\mathbb R}    
\newcommand{\C}{\mathbb C}    
\newcommand{\PP}{\mathbb P}   
\newcommand{\T}{\mathbb T}    
\newcommand{\N}{\mathbb N}  

\newcommand{\F}{\mathcal F} 

\newcommand{\suchthat}{\ : \ }
\newcommand{\<}{\langle}   
\renewcommand{\>}{\rangle} 

\newcommand{\Log}{\operatorname{Log}}

\newcommand{\am}{{\mathcal{A}}}
\newcommand{\ignore}[1]{\relax}
\newcommand{\dd}{\partial}

\newcommand{\Hom}{\operatorname{Hom}}

\newcommand{\OS}{\operatorname {OS}}  
\newcommand{\Gr}{\operatorname{Gr}}
\newcommand{\gys}{\operatorname{Gys}}
\newcommand{\Star}{\operatorname{Star}}
\newcommand{\Link}{\operatorname{Link}}
\newcommand{\shad}{\operatorname{res}}

\newcommand{\Vol}{\operatorname{Vol}}
\newcommand{\ind}{\operatorname{ind}}
\newcommand{\id}{\operatorname{id}}

\newcommand{\cp}{{\mathbb C}{\mathbb P}}
\newcommand{\tp}{{\mathbb T}{\mathbb P}}

\newcommand{\LL}{\mathcal L}

\begin{document}

\title{Tropical Homology}

\author{Ilia Itenberg}
\address{Institut de Math\'{e}matiques de Jussieu--Paris Rive Gauche\\
Sorbonne Universit\'{e}\\  
4 place Jussieu, 75252 Paris Cedex 5, France \\
and 
D\'{e}partement de Math\'{e}mati\-ques et Applications,
Ecole Normale Sup\'{e}rieure \\
45 rue d'Ulm, 75230 Paris Cedex 5, France} 
\email{ilia.itenberg@imj-prg.fr}
\author{Ludmil Katzarkov}
\address{Universit\"at Wien, Department of Mathematics, Vienna, Austria and University of Miami, Department of Mathematics, Miami, FL, USA}
\email{lkatzarkov@gmail.com}
\author{Grigory Mikhalkin}
\address{Universit\'e de Gen\`eve,  Math\'ematiques, Villa Battelle, 1227 Carouge, Suisse}
\email{grigory.mikhalkin@unige.ch}
\author{Ilia Zharkov}
\address{Kansas State University, 138 Cardwell Hall, Manhattan, KS 66506 USA}
\email{zharkov@math.ksu.edu}

\begin{abstract}
Given a tropical variety $X$ and two non-negative integers
$p$ and $q$ we define a homology group $H_{p,q}(X)$ which is a 
finite-dimensional vector space over $\Q$. 
We show that if $X$ is a smooth tropical variety that can be represented as the tropical limit
of a 1-parameter family of complex projective varieties, then $\dim H_{p,q}(X)$ coincides with the Hodge number $h^{p,q}$ of a general member of the family. 
\end{abstract}
\maketitle

\footnote{The research was partially supported by the NSF FRG grants DMS-0854989 and DMS-1265228.
L.K. was  supported by Simons research grant, NSF DMS-150908, ERC Gemis, DMS-1265230,
DMS-1201475, OISE-1242272 PASI and Simons collaborative Grant HMS.
Research of G.M. was partially supported by the grant TROPGEO of the European Research Council,
and by the grants 140666, 141329, 159240, 159581 and NCCR ``SwissMAP"
of the Swiss National Science Foundation.}

\section{Introduction}
\subsection{Homology theory for tropical varieties}
Tropical varieties are certain finite-dimen\-sional polyhedral complexes enhanced 
with the {\em tropical structure}. This is a geometric structure that can be thought of
as a version of an affine structure for polyhedral complexes.
For example, the tropical projective space $\tp^N$ is a smooth projective tropical variety
homeomorphic to the $N$-simplex. The restriction of the tropical structure to
the relative interior of a $k$-dimensional face $\sigma$ of $\tp^N$ turns $\sigma$ into 
$\R^k$ (with the tautological 
affine structure of $\R^k=\Z^k\otimes\R$). 
A projective tropical $n$-variety $X$ is a certain $n$-dimensional
polyhedral complex in $\tp^N$.

A tropical structure on $X$ can be used to define a
natural coefficient system $^\Z\F_p$. This system
is not locally constant everywhere,
but it is constant on the relative interiors of faces 
of $X$. Furthermore,
it is a constructible cosheaf of abelian groups.
The tropical $(p,q)$-homology group $H_{p,q}(X)$
is the $q$-dimensional homology group of $X$ with coefficients in $\F_p= \ ^\Z\F_p\otimes\Q$. 

An important example of projective tropical varieties
is provided by the {\em tropical limit} of an algebraic family
$Z_w \subset\cp^N$, $w\in\C$, $t=|w|\to \infty$, of complex projective $n$-dimensional varieties.
It may be shown (cf. e.g. the fundamental theorem
of tropical geometry of \cite{MS})
that the sets $\Log_{t}(Z_w)\subset\tp^N$ 
converge to an $n$-dimensional balanced weighted polyhedral complex $X$ in $\tp^N$.
If $X$ is a smooth tropical variety, then for a generic $w$ the complex variety
$Z_w$ is smooth.
The main result of this paper establishes the equality between $\dim H_{p,q}(X)$ and the Hodge numbers $h^{p,q}(Z_w)$.

\begin{theorem}[Main Theorem]\label{thm:main} 
Let $\mathcal Z \subset \C\PP^N \times \mathcal D^*$ be a complex analytic one-parameter family of projective varieties over the punctured disc $\mathcal D^*$. Assume that $\mathcal Z$ admits a tropical limit $X\subset \T\PP^N$, which is a smooth projective
$\Q$-tropical variety (see Definition \ref{def:qtrop}).
Then, the dual spaces $\Hom (H_q(X;\F_p), \Q)$ to the tropical homology groups $H_q(X;\F_p)$
are naturally isomorphic to the associated graded pieces $W_{2p}/W_{2p-1}$
of the weight filtration in the limiting mixed Hodge structure
on $H^{p+q} (\mathcal Z_\infty, \Q)$, where $\mathcal Z_\infty$ is the canonical
fiber of the family $\mathcal Z$.
\end{theorem}

Under the assumptions of Theorem \ref{thm:main},
the limiting mixed Hodge structure is of Hodge-Tate type.
That is, only even associated graded pieces 
$\Gr^W_{2p} H^k (\mathcal Z_\infty ; \Q)=W_{2p}/W_{2p-1}$
are non-trivial and they have Hodge $(p,p)$-type. Hence, the dimensions of the pieces in the Hodge filtration on  $H^{k} (\mathcal Z_\infty; \Q)$ can be recovered from the weight filtration. On the other hand, $\dim F^p H^k (Z_w, \Q) = \dim F^p H^k (\mathcal Z_\infty, \Q)$, cf \cite{Clemens}. 
Thus, we can conclude that the Hodge numbers $h^{p,q}(Z_w)$ 
agree with the dimensions of the spaces $\Gr^W_{2p} H^{p+q} (\mathcal Z_\infty; \Q)$.

\begin{corollary}\label{cor:main}
The Hodge numbers $h^{p,q} (Z_w)$ of a general fiber equal the dimensions of the tropical homology groups $H_q(X;\F_p)$.
\end{corollary}

Note that different choices
of the central fiber for $\mathcal Z$
may lead to different triangulations of 
the polyhedral complex $X$.
Meanwhile $\F_p$, $\F^p$, $H_{p,q}(X)$ and
$H^{p,q}(X)$ are central fiber choice free as
none of them makes use of such a choice.
Furthermore,
$\F_p$ and $\F^p$ are equally well defined for an arbitrary tropical variety $X$ 
even if it cannot be presented as the tropical limit of a family of complex varieties.
The groups $H_q(X;\F_p)$ and $H^q (X; \F^p)$ give a homology theory in tropical geometry.

We prefer to work with homology rather than with cohomology because the tropical homology have a more transparent geometric meaning. For cohomology theory one can consider the constructible {sheaf} $\F^p$ on $X$ whose stalks are dual to the spaces $\F_p$.
Differential forms and currents may also be considered
on tropical varieties, see \cite{Lagerberg}.
Note that the work
\cite{ChambertLoir-Ducros} makes use of the pull-backs of such forms to Berkovich spaces.
A recent work \cite{Jell-Shaw-Smacka}
provides a link between usage of such differential forms and
the tropical homology definition considered in this paper.

We prove Theorem \ref{thm:main} by providing a quasi-isomorphism between the tropical cellular complexes and the dual row complexes of the $E^1$-term of the weight spectral sequence for the limiting mixed Hodge structure (see Theorem \ref{theorem:main}).  

\begin{remark}
If $X$ comes as the tropical limit of complex varieties $Z_w$, 
then the geometric meaning of the tropical coefficients $\F_p$ and $\F^p$ originates 
from the tropical collapse map  $\pi: Z_w \to X$. This map comes from logarithmically  mapping  $Z_w$ to its amoeba 
and then collapsing the amoeba to $X$. 
The sheaf $\F^p$ can then be identified with the direct image $R^p\pi_* \underline{\Q}$, 
and the Leray spectral sequence which calculates $H^{p+q} (Z_w; \Q)$ 
has the second term $E_2^{q,p}=H^q (X; \F^p)$. 
\end{remark}

\subsection{Tropical Euler characteristics}
Each coefficient system $\F_p$ independently
gives homology groups $H_q(X;\F_p)$ 
for all dimensions $q$.
The corresponding Euler characteristic 
\begin{equation}\label{trop-chi_p}
\chi_p(X)=
\sum\limits_{q=0}^n (-1)^q \dim H_q(X;\F_p) 
\end{equation}
is a basic invariant of the tropical variety
$X$ which is especially easy to compute.
The corresponding classical invariants 
\begin{equation}\label{class-chi_p}
\chi_p(Z_w)=\sum\limits_{q=0}^n (-1)^q
h^{p,q}(Z_w)
\end{equation}
were introduced by Hirzebruch \cite{Hirz}
in the form of $\chi_y$-genus
$\chi_y(Z_w)=\sum\limits_{p=0}^n
\chi_p(Z_w)y^p.$ 
Clearly, Theorem 1 implies that 
\begin{equation}\label{chiravenstvo}
\chi_p(X)=\chi_p(Z_w).
\end{equation}
E.g. $\chi_0(X)$ is nothing else but
the conventional Euler characteristic of $X$,
while $\chi_0(Z_w)$ is the holomorphic 
Euler characteristic (arithmetic genus)
of $Z_w$.
As usual for the Euler 
characteristic, the alternation of signs
in (\ref{trop-chi_p}) provides additional invariance
properties. 

Since the definition of the limiting Hodge structure, 
there was developed a way to compute it with the help
of a central fiber $Z_0$ of the family $\mathcal Z$,
{\it i.e.}, through its extension over ${\mathcal D}\supset \mathcal D^*$, 
see \cite{Steen} and references therein.
Note that the choice of $Z_0$ is not unique,
and different choices yield different homology data
of $Z_0$.

In the same time, the K\"ahler manifolds $Z_w$, $w\neq 0$, are 
symplectomorphic,
and thus have 
the same topological homology groups. 
One may expect that
it should be possible to find
such a central fiber
$Z_{\operatorname{nearby}}$
(instead of $Z_0$)
that 
$Z_{\operatorname{nearby}}$ is symplectomorphic
to $Z_w$, $w\neq 0$.
The problem is that 
(in the case of non-trivial family
${\mathcal Z}$)
such $Z_{\operatorname{nearby}}$
cannot carry a complex structure in the conventional
sense
(the holomorphic tangent subbundle  
$T^{1,0}(Z_{\operatorname{nearby}})$
in $T(Z_{\operatorname{nearby}})\otimes\C$
cannot stay transversal
to $T(Z_{\operatorname{nearby}})\otimes\R$ 
due to behavior of $T^{1,0}(Z_w)$ for small $w\neq 0$).

The notion of {\em motivic nearby fiber} 
(see \cite{DenefLoeser}, \cite{Bittner} and
\cite{Steen}) avoids this
problem by defining 
a class $\psi$
in the Grothendieck ring of 
varieties 
so that it should correspond to 
$Z_{\operatorname{nearby}}$ (would it
exist as a variety). 
The class $\psi$ can be expressed as a
certain linear combination of strata of 
$Z_0$
and 
does not depend on
ambiguity in the choice of $Z_0$
thanks to the alternation of signs in
the expression of $\psi$ via $Z_0$. 

When we pass from $Z_w$ 
to the motivic nearby fiber $\psi$ 
we loose some homological 
information. 
For example, $\psi$ is the class of the empty set 
for the degeneration corresponding to a smooth
tropical elliptic curve (or, more generally,
a smooth tropical Abelian variety).
Under the assumption of 
Theorem \ref{thm:main}, the motivic nearby fiber $\psi$ 
is a linear combination 
of the powers of the class ${\mathbb L}$ of the affine line, and thus carries
the same amount of
data as the $E$-polynomial of Deligne-Hodge. 
\begin{corollary}
Under the assumptions of Theorem 1 we have 
\begin{equation}\label{epsi}
E(\psi)=E({\mathcal Z}_\infty)
=\sum\limits_{p=0}^n
\chi_p u^pv^p.
\end{equation}
\end{corollary}
\begin{proof}
The second equality of \eqref{epsi}
is the combination of
the definition of the $E$-polynomial
and Theorem 1.
The first equality follows
from the additivity of $E$-polynomial
with the help of 
the 
description of central fiber
from 
Proposition
\ref{thm:snc}
and the well-known $E$-polynomial
computation for the hyperplane
arrangement complements
(cf. Theorem \ref{OS}). 
For the first equality see also section 11.2.7 of \cite{Steen}.
\end{proof}

In general, $E(\psi)$ 
does not determine individual numbers
$h^{p,q}(Z_w)=\dim H_q(X;\F_p)$.
In the special case 
of complete intersections,
the Lefschetz hyperplane section theorem
determines $h^{p,q}(Z_w)$ for $p+q\neq n$,
and thus $E(\psi)$ (which can easily be
read from the combinatorial data of $X$)
suffices to recover $h^{p,q}(Z_w)$.
This observation appeared in \cite{KatzStapledon}
for the case 
of hypersurfaces. 

\begin{rmk}
Note that the tropical limit of ${\mathcal Z}$
introduced in Definition
\ref{fine-conv} does not require to make
any choice for the central fiber $Z_0$
whatsoever.
Different choices of $Z_0$
correspond to different triangulations
of $X$ while the homology groups
$H_q(X;\F_p)$
do not require
introduction of such an additional structure.
\end{rmk}

\begin{rmk}
According to Theorem \ref{thm:main}, 
the Hodge
numbers of an $n$-dimensional projective variety 
$\mathcal Z_w\subset\cp^N$ can be recovered
from its tropical limit $X$ 
provided that 
$X$ is a smooth regular projective $\Q$-tropical
variety.
In this case,
since $\F_0=\Q$ 
is 
a constant cosheaf (see Example \ref{examples:tropical-homology1}), 
the dimensions of homology groups of $X$ with rational coefficients 
are determined by the Hodge numbers $h^{p,q}(Z_w)$.
For example, if $h^{0,n}(Z_w)=h^{0,0}(Z_w)=1$
while $h^{0,k}(Z_w)=0$ for $k=1,\dots,n-1$
(which is the case e.g. for a certain class 
of Calabi-Yau $n$-folds such as the K3-surfaces for $n = 2$), the topological space 
$X$ (which does not have to be a manifold)
is a rational homology sphere. 
\end{rmk} 

\begin{acknowledgments}
We 
are grateful to Sergey Galkin and Luca Migliorini for 
useful discussions and explanations.
The present work started during the fall 2009 semester ``{\it Tropical geometry}" at MSRI,
and we would like to thank the MRSI for hospitality and excellent working conditions.
\end{acknowledgments}

\section{Tropical varieties and their homology}
Main results in this paper concern smooth 
$\Q$-tropical varieties embedded in some projective space. 
In this section we adapt definitions of tropical geometry to this special case. The notion of tropical homology can be defined for more general tropical spaces, e.g. singular or non-compact, see \cite{MZh} but its relation to the (mixed) Hodge structures of singular complex varieties
is to be understood. 

\subsection{Polyhedral complexes in tropical projective space}\label{section:polyhedral}

The tropical affine space $\T^N$ is the topological space $[-\infty,+\infty)^N$ (homeomorphic to the
$n$-th power of a half-open interval) enhanced with an integral affine structure defined as follows. We stratify the space $\T^N$ by
$$\T^\circ_I:= \{y=(y_1,\dots,y_N)\in \T^N 
\ :\ y_i=-\infty, i\in I \ \text{ and } \ y_i>-\infty, i\notin I \} \cong \R^{N-I},
$$ 
where $I$ runs over subsets of $\{1,\dots,N\}$.
We set $\T_I \cong \T^{N-I}$ to be the closure of $\T^\circ_I$ in $\T^N$.
On each $\T^\circ_I$ the integral affine structure
is induced from $\R^{N-I}$, and for pairs $I\subset J$ the projection maps $\T^\circ_I \to \T^\circ_J$ are  $\Z$-affine linear.
Here and later we use $N-I$ in the exponent to denote the product of $N-|I|$ factors in the complement of the subset $I$ in $\{1,\dots,N\}$.

Let $B_I$ be a $(k-|I|)$-dimensional ball in $\T^\circ_I$. A $k$-dimensional {\em $I$-ball} $B$ in $\T^N$ is the $\epsilon$-neighborhood of $B_I$ in $\T^N$ (using the projection map $\pi^\T_I: \T^N \to \T_I$): 
\begin{equation}\label{eq:i-ball}
B=\{y\in \T^N \suchthat \pi^\T_I(y) \in B_I \text{ and } y_i < \log \epsilon , i \in I\}
\end{equation}  
for some $\epsilon >0$. In particular, an $\emptyset$-ball is just an ordinary $k$-dimensional ball in $\R^N \subset \T^N$. For $I\neq \emptyset$ an $I$-ball is a $k$-dimensional manifold with corners.
The {\em boundary} $\dd B$ of an $I$-ball $B$ is defined as $\overline B \setminus B$, that is, we exclude from its topological boundary all strata at infinity, and then take the closure in $\T^N$.

The tropical projective space $\T\PP^N$ can be defined as the quotient of $\T^{N+1}\setminus (-\infty, \dots, -\infty)$ by the equivalence $(x_0,\dots x_N)\sim (x_0+\lambda, \dots x_N+\lambda)$ for any $\lambda\in \R$. In particular, as a topological space, $\T\PP^N$ is homeomorphic to an $N$-simplex.
Alternatively, $\T\PP^N$
can be glued from $N+1$ affine charts $U^{(i)} =\{x_i\ne -\infty\} \cong \T^N$ with coordinates $y^{(i)}_k= x_k-x_i, i\ne k$.
Every two charts $U^{(i)}$ and $U^{(j)}$
are identified along $U^{(i)} \cap U^{(j)} \cong \T^{N-1}\times \R$ via
$y^{(i)}_k= y^{(j)}_k - y^{(j)}_i$ for $k\ne i, j$, and $y^{(i)}_j= - y^{(j)}_i$.

For any subset $I \subset  \{0,\dots N\}$ we denote by $\T\PP_I\cong \T\PP^{N-I}$ and by $\T\PP^\circ_I\cong \R^{N-I}$ the closed and open coordinate strata of $\T\PP^N$, respectively. That is $\T\PP_I$ is defined by setting $x_i=-\infty$ for $i\in I$, and for $\T\PP^\circ_I \subset \T\PP_I$ we additionally require $x_i\ne-\infty$ for $i\not\in I$.

The directions parallel to the $j$-th coordinate in $\R^N\subset\T^N$ towards its $-\infty$-value are 
called {\em divisorial directions}. The primitive integral vector along a divisorial direction
(pointing towards $-\infty$ as the direction itself) is called a {\em divisorial vector}.
The positive linear combinations of divisorial vectors in $I\subset \{1, \dots, N\}$ span the $I$-th {\em divisorial cone} in $\R^N$. 

The notions of divisorial directions, vectors and cones are well defined for all open strata in $\T\PP^N$. Indeed, if $\T\PP^\circ_I \subset \T\PP^N$ is such a stratum and $j\not\in I$, we can take any chart $U_I^{(i)}\cong \T^{N-I}$ of $\T\PP_I$ such that $i\ne j$ and define the $j$-th divisorial direction as above. Clearly, the $j$-th divisorial directions agree in any two such charts. The same can be said about the $J$-th divisorial cones for any subset $J$ (of size not greater than $N-I-1$) disjoint  from $I$.

Recall that an $n$-dimensional polyhedral complex
 $Y^\circ\subset\R^N$
with rational slopes 
 is a finite union of
$n$-dimensional convex polyhedral domains
called {\em facets}.
Each facet is the intersection of a finite number of half-spaces
of  the form $mx\le a$, 
where $x\in \R^N$, $a\in\R$, $m\in \Z^N$. 
The intersection of any number of facets 
is required to be their common face.


\begin{lemma}\label{lemma:closure}
The closure in $\T\PP^N$ of an $n$-dimensional polyhedral complex $Y^\circ\subset\R^N$ with rational slopes intersects each stratum $\tp^\circ_I$ in a subset which supports a polyhedral complex of dimension $\le (n-1)$ with rational slopes. 
\end{lemma}

\begin{proof}
It is enough to show that the closure in $\T^N$ of a polyhedral domain $D\subset \R^N$ intersects a stratum $\T^\circ_I$ in a polyhedral domain $D_I$ of smaller dimension. Let $\bar D_I$ be the image of $D$ under the projection along the divisorial directions in $I$. Clearly, $\bar D_I$ is a polyhedral domain in $\T^\circ_I$.

Consider the intersection of the $I$-th divisorial cone in $\R^N$ with the asymptotic cone of $D$
(that is, the cone formed by the vectors $v \in \R^N$ with the property: $x\in D$ implies that $x + av \in D$
for any $a \geq 0$).
Then, observe that $D_I= \bar D_I$ if this intersection contains a ray (in that case the dimension of $D_I$ is smaller than $n$),
and $D_I$ is empty otherwise.
\end{proof} 

An $n$-dimensional polyhedral complex
 $Y^\circ\subset\R^N$
with rational slopes is called {\em weighted} if the facets
are equipped with non-negative integers.
Recall the {\em balancing condition}:
for every face $\Delta$ of codimension 1 the weighted sum 
of primitive (relatively to $\Delta$)
integer outward tangent vectors
in the facets incident to $\Delta$
should be parallel to $\Delta$.

\begin{definition}\label{defn:polycomplex}
{\em A {\rm (}weighted{\rm )} polyhedral complex $Y\subset\T\PP^N$} of dimension at most $n$
is a finite union of {\rm (}weighted{\rm )}  polyhedral complexes $Y_I$ of dimensions $\le n$ in 
$\tp^\circ_I$, where $I$ runs over the subsets of $\{0, \dots, N\}$, such that for any pair $I \subset J$ the intersection of the closure of any face of $Y_I$ with $\tp^\circ_J$ is a face of $Y_J$. 
The complex $Y$ is {\rm(}pure{\rm)} $n$-dimensional if any point in $Y$ lies in the closure of some $n$-dimensional face. It is {\em balanced} if all $n$-dimensional complexes in the union satisfy the balancing condition. 
\end{definition} 

By a {\em face} of a polyhedral complex $Y \subset \T\PP^N$ 
we mean the closure in $Y$ of a face of a complex from the union. Any polyhedral complex $Y$ can be 
considered weighted by setting all weights equal to one.
By default, we mean this situation unless other weights
are explicitly prescribed. 

\subsection{Smooth projective tropical varieties}\label{sec:varieties}
Now we define a more restrictive class of polyhedral complexes in $\T\PP^N$.
A {\em convex regular $\Q$-polyhedral domain} $D$ in $\T^{N}$ is the intersection
of a finite collection of half-spaces $H_k$ of the form
\begin{equation}\label{H_k}
H_k=\{x\in\T^N\ |\ mx\le a\}\subset\T^{N}
\end{equation}
for some $m\in\Z^N$ and $a\in\Q$. Here, we assume that if some component $m_i$ of $m$ is negative,
the corresponding component $x_i$ of $x$
can not take the value $-\infty$.
So that $H_k$ only contains points $x\in \T^n$ for which the scalar product $m x$ is well-defined.
The following statement is immediate.

\begin{lemma}\label{lemma:precylinders}
Let $D \subset \T^N$ be a non-empty convex regular $\Q$-polyhedral domain defined by the inequalities $m^{(r)}x \le a^{(r)}$. Let $I$ be a subset of $\{1,\dots, N\}$.  Then $D_I:= D\cap \T_I$ is  non-empty if and only if $m^{(r)}_i\ge 0$ for all $r$ and all $i\in I$. \qed
\end{lemma}

The boundary $\dd H_k$ of a half-space $H_k$ is given by the equation $mx=a$.
A {\em mobile face} $E$ of $D$ is the intersection
of $D$ with the boundaries of some of its defining half-spaces given by \eqref{H_k}.
The adjective {\em mobile} stands here to distinguish such faces from
more general faces which we define below and
which are allowed to have support in $\T^N\setminus\R^N$, {\it i.e.}, to be disjoint from
$\R^N\subset \T^{N}$. (Such faces disjoint from
$\R^N\subset \T^{N}$ have reduced mobility and are called {\em sedentary}.)

The {\em dimension} of a convex regular  $\Q$-polyhedral domain $D$ is its dimension as a topological manifold (possibly with boundary).
Observe that for each non-empty mobile face $E$ of $D$ the intersection $E^\circ=E\cap\R^N$ is non-empty. 
Each mobile face of $D$ is a convex regular $\Q$-polyhedral domain itself.

\begin{definition}\label{def:poly_complex}
An {\em $n$-dimensional regular $\Q$-polyhedral complex}
$Y=\bigcup D\subset \T^{N}$
is the union of a finite collection of $n$-dimensional convex regular $\Q$-polyhedral domains $D$, called
the {\em facets} of $Y$, subject to the following property:
for any collection $\{D_j\}$ of facets, their intersection $\bigcap D_j$ is a mobile face of each facet $D_j$.
Such intersections are called the mobile faces of $Y$.
\end{definition}

For mobile faces $E$ of $Y$ and subsets $I\subset\{1,\dots,N\}$, it is convenient  to treat
the intersections $E \cap\T_I$ also as faces (at infinity) of $Y$.

\begin{definition}\label{def:sedentarity}
Let  $E$ be a mobile face of $Y$, and let $I$ be a subset of $\{1,\dots,N\}$.
We say that
$$E_I=E \cap\T_I$$ is a {\em face} of $Y$.
The {\em sedentarity} of the face $E_I$ is $s=|I|$,
while its {\em refined sedentarity} is $I$.
We call the face $E$ a {\em parent} of $E_I$.
The poset $\Pi(E)$ of faces $E_I$, where $I$ runs over all subsets $I\subset \{1,\dots,N\}$, is called the {\em family} of $E$.
\end{definition}

Clearly, the {mobile} faces are the faces of sedentarity $0$. 
A mobile face $E$ of $Y$ such that $E_I$ is non-empty for some subset $I \subset \{1, \ldots, N\}$ contains a ray along the $j$-th divisorial direction for each $j\in I$ (cf. Lemma \ref{lemma:precylinders}). 
In this case we say that this is a {\em divisorial direction of the face} $E$, see Figure \ref{mobile}.
By convexity this also imply that $E$ contains the entire $I$-th divisorial cone. 
(Note that 
a more general polyhedral complex considered in Section \ref{section:polyhedral} 
may have a face whose closure 
intersects a stratum $\T_I$, but the face does not contain the $I$-th divisorial cone.) 
Sedentary faces of $Y$ are mobile when considered
in the respective strata of $\T^N$, and as such also have divisorial   
directions defined. 

The following lemma describes the geometry of a regular $\Q$-polyhedral complex $Y$ near its sedentary faces.

\begin{lemma}\label{lemma:cylinders}
Let $Y$ be a regular $\Q$-polyhedral complex in $\T^N$. Let $I$ be a subset of $\{1,\dots, N\}$  such that $Y_I:= Y \cap \T_I$ is non-empty.
Then, $Y_I$ is a regular $\Q$-polyhedral complex in $\T_I$. Moreover, its regular neighborhood
$$Y_I^\epsilon:= \{y\in Y \suchthat y_i< \log \epsilon, i\in I\},
$$
for sufficiently small $\epsilon > 0$, splits as the product
$$Y_I^\epsilon = Y_I \times \T^{I}_\epsilon,
$$
where  $\T^{I}_\epsilon := \{x_i< \log \epsilon, i\in I\} \subset \T^I$.

\end{lemma}
\begin{proof}
Let $E$ be a mobile face of $Y$ such that $E_I:=E\cap \T_I$ is non-empty. 
By Lemma \ref{lemma:precylinders}
the defining inequalities  \eqref{H_k} for $E$ must have $m_i \ge 0$ for all $i\in I$. Furthermore, 
to define the face $E_I \subset \T_I$ one can take those inequalities for $E$ that have $m_i=0$ for all $i\in I$.
Each of the remaining inequalities for $E$ have $m_i>0$ for at least one $i\in I$,
and hence they are satisfied for sufficiently small $y_i, i\in I$.

For a pair of mobile faces $E,F$ of $Y$ we have $(E\cap F)_I = E_I \cap F_I$ (both are defined by 
plugging $y_i = -\infty$, $i \in I$, into the union of the defining inequalities for $E$ and $F$). Thus, the statement for a regular neighborhood of $Y_I$ follows from the statement for each individual mobile face of $Y$.
\end{proof}

\begin{figure}
\centering
\includegraphics[height=45mm]{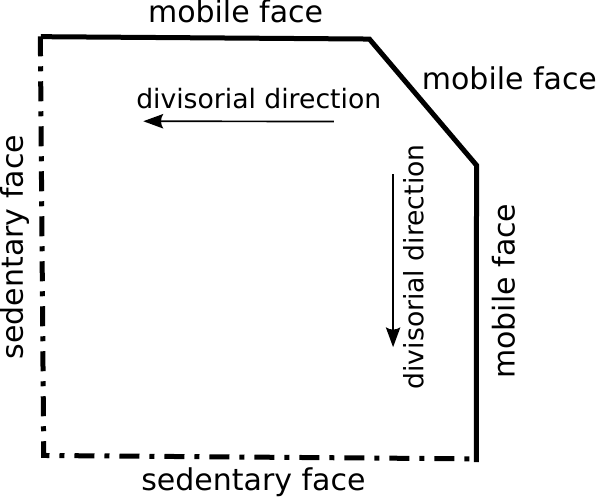}
\caption{Mobile and sedentary faces of a regular $\Q$-polyhedral domain in $\T^N$.}
\label{mobile}
\end{figure}

Let $E\subset Y$ be a mobile face of $Y$, and let $x$ be a point in the relative interior of $E$.
Consider the tangent cone $T_x Y$ to $Y$ at $x$, 
and denote by $\Sigma_E$ the quotient of $T_x Y$ 
by the linear span $L_E$ of $E\cap\R^N$. 
This quotient is a {\it $\Q$-polyhedral fan} in the vector space $\R^N/L_E$,
{\it i.e.}, a polyhedral complex with rational slopes which has a cone structure in $\R^N/L_E$. 
The fan $\Sigma_E$ is called {\em relative fan} of $Y$ at $E$. 

Finally, we recall the tropical notion of smoothness.
A \emph{matroid} $M=(M,r)$ is a finite set $M$ together with a rank function $r:2^M\to\Z_{\ge0}$ such that
we have the inequalities
$r(A\cup B)+r(A\cap B) \le r(A) + r(B)$ and $r(A)\le |A|$,
for any subsets $A,B\subset M$,
as well as the inequality $r(A)\le r(B)$ whenever $A\subset B$.
Subsets $A \subset M$ such that $r(A) = |A|$ (respectively, $r(A) < |A|$) are called {\it independent}
(respectively, {\it dependent}). 
Subsets $F\subset M$ such that $r(A)>r(F)$ for any 
$A \supsetneq F$ 
are called {\em flats}
of $M$ of rank $r(F)$.
Matroid $M$ is  {\em loopless} if $r(A)=0$ implies $A=\emptyset$.

The Bergman fan of a loopless matroid $M$ is a $\Q$-polyhedral fan $\Sigma_M\subset \R^{|M|-1}$ constructed as follows
(see (\cite{AK}). Choose $|M|$ integer vectors $e_j\subset\Z^{|M|-1}\subset\R^{|M|-1}$, $j\in M$,
such that $\sum\limits_{j\in M} e_j=0$ and any $|M|-1$ of these vectors form a basis of $\Z^{|M|-1}$.
To any flat $F\subset M$ we associate a vector
$$e_F:=\sum\limits_{j\in F} e_j\in\R^{|M|-1}.
$$
E.g, $e_M=e_\emptyset=0$, but $e_F\neq 0$ for any other (proper) flat $F$.
To any flag of flats $F_{i_1}\subset\dots\subset F_{i_k}$
we associate the
convex cone generated by $e_{F_{i_j}}$. We define $\Sigma_M$ to be the union of such cones, which is, clearly, an $(r(M)-1)$-dimensional rational simplicial fan.
The fan $\Sigma_M$ is 
called the {\em Bergman fan} of $M$.

\begin{definition}
A regular $\Q$-polyhedral complex in $\T^N$ is {\it smooth} 
at a mobile face $E \subset Y$ if the relative fan $\Sigma_E$
has the same support as the Bergman fan $\Sigma_M$ for some loopless matroid $M$
(recall that the support of a fan is the union of its cones).
A  regular $\Q$-polyhedral complex $Y\subset \T^N$ is called {\em smooth},
if it is smooth at all
its mobile faces.
\end{definition} 

Note that since all Bergman fans are balanced (cf. \cite{AK}), a smooth regular $\Q$-polyhedral complex $Y$ is automatically balanced. That is, every mobile face of $Y$ of codimension $1$ has a balanced relative fan:
the sum of the outward primitive integer vectors along its rays is zero. 

\begin{definition}\label{def:qtrop}
A closed subset $X\subset \T\PP^N$ is
a {\em smooth regular projective $\Q$-tropical variety}
if on every affine chart $U^{(i)}\cong \T^N \subset \T\PP^N$ 
it restricts to a smooth regular
 $\Q$-polyhedral complex $X^{(i)}$ in $\T^N$.
\end{definition}

By definition, the faces of $X$ are the closures of the faces of $X^{(i)}$ in $X$.
A face $\Delta$ of $X$ is determined
by its relative interior which coincides
with the relative interiors
for all non-empty restrictions
to the complexes $X^{(i)}$.

For convenience we make the following additional assumption:
each face $\Delta$ of $X$ lies entirely in at least one chart $X^{(i)}$.
For instance, we avoid considering $X=\T\PP^N$ as just one face, it needs some subdivision.
We can always subdivide
$X$ in order to achieve this requirement. 

Sedentarity in $\tp^N$ is inherited from the
projection map $\T^{N+1}\setminus\{0\}\to\tp^N$.
Recall that for any face $\Delta$ of a smooth regular
projective $\Q$-tropical variety $X \subset \T\PP^N$
the divisorial directions are intrinsically defined.
The splitting of the boundary neighborhoods $X_I^\epsilon = X_I \times \T^{I}_\epsilon$,
however, is not canonical,
it depends on the chart.

We introduce some convenient notations and terminology.
We write $\Delta' \prec^s_j \Delta$ (and $\Delta \succ^s_j \Delta'$) when $\Delta'$ is a face of $\Delta$ of codimension $j$ and cosedentarity $s$ (that is,
the sedentarities of $\Delta'$ and $\Delta$ differ by $s$).
We omit the superscript $s$ in case $s=0$ and simply write $\Delta' \prec_j \Delta$.
A face $\Delta$ of $X$ is called {\em infinite} if it has a subface of higher sedentarity.
Otherwise, $\Delta$ is called {\em finite} (the sedentarity of $\Delta$ may be positive).
The {\em star} of a face $\Delta$ of $X$ is the poset
formed by the faces of $X$ that contain $\Delta$ and have the same sedentarity as $\Delta$. 

Lemmas  \ref{lemma:precylinders} and \ref{lemma:cylinders} imply the following statement. 

\begin{proposition}\label{prop:poset}
Let $X$ be a smooth regular projective $\Q$-tropical variety, and let $\Delta_0$ be a mobile face of $X$.
\begin{enumerate}
\item Any face of $X$ belongs to a single family. 

\item The face $\Delta_0$ contains a unique subface
of maximal sedentarity; denote this surface by $\Delta_J$, where $J$ is its refined sedentarity. The face $\Delta_J$ is finite.
 
\item The family $\Pi(\Delta_0)$ is the rank $|J|$ lattice {\rm (}under $\prec_j^j${\rm )} isomorphic 
to the lattice of all subsets of $J$. 

\item The asymptotic cone of $\Delta_0$ coincides with 
the divisorial cone of $\Delta_0$. The divisorial directions of $\Delta_0$ are indexed by the elements of $J$.
\item All faces 
in the family $\Pi(\Delta_0)$ have isomorphic stars. A relation $\Delta'\prec_j\Delta''$ for any two faces of $X$ 
also holds for their parents: $\Delta'_0\prec_j\Delta''_0$. 
In addition, a relation $\Delta'_0\prec_j\Delta''_0$ among mobile faces 
gives rise to an injection of $\Pi(\Delta'_0)$ into $\Pi(\Delta''_0)$. 
\qed 
\end{enumerate} 
\end{proposition}

\subsection{Local homology and Orlik-Solomon algebra}\label{sec:Orlik-Solomon} 
Let $\Sigma=\bigcup \sigma\subset \R^N=\Z^N\otimes \R$ be a $\Q$-polyhedral fan.
For each cone $\sigma\subset \Sigma$, we denote by $\<\sigma\>_\Z$
the integral lattice in the vector subspace linearly spanned by $\sigma$. 

\begin{definition}\label{def:local_fk}
The {\em homology} group $^\Z\F_k(\Sigma)$ is the subgroup of $ \wedge^k\Z^N$ generated by the elements $v_1\wedge\dots \wedge v_k$, where all $v_1,\dots,v_k \in \<\sigma\>_\Z$ for some cone $\sigma\in\Sigma$.  It is important that all $k$ vectors $v_i$ come from the {\em same} cone. The {\em cohomology} is the dual group $^\Z\F^k(\Sigma):=\Hom(^\Z\F_k(\Sigma),\Z)$, 
which is the quotient of $\wedge^k(\Z^N)^*$  by $(^\Z\F_k(\Sigma))^\perp$. 
\end{definition} 

It is not hard to see (cf. \cite{Zh}) that the cohomology groups form a graded algebra $^\Z\F^\bullet (\Sigma)$ over $\Z$ under the wedge product in $\wedge^k(\Z^N)^*$. 

We restrict our attention to the case where $\Sigma$ is the Bergman fan $\Sigma_M$ associated to a loopless matroid $M$. 
On the other hand, to any loopless matroid $M$ one can also associate its {\it Orlik-Solomon algebra} $\OS(M)$ 
as follows (see, e.g. \cite{OT}). 

Let $W$ be a rank $N+1$ free abelian group generated by elements $f_0,\dots,f_N$, 
where $| M | = N + 1$.
Then, $\OS^\bullet(M):=\wedge^\bullet W/{\mathcal I}^\bullet$, where the {\it Orlik-Solomon ideal} $\mathcal I$
is generated by the elements 
$$\partial(f_{i_0}\wedge f_{i_1}\wedge \dots \wedge f_{i_k}):=\sum_{s=0}^k (-1)^s f_{i_0}\wedge \dots \hat f_{i_s} \dots \wedge f_{i_k},
$$
for all dependent subsets $I=\{{i_0},{i_1}, \dots ,{i_k}\}$ of the matroid $M$.

More relevant for us is the {\it projective Orlik-Solomon algebra} $\OS_0^\bullet(M)$,
which is the following modification of $\OS(M)$. Let $W_0$ be the subgroup of $W$ generated by all differences $f_i-f_j$. Then, we set $\OS_0^\bullet(M):=\wedge^\bullet W_0/{\mathcal I}^\bullet_0$, where ${\mathcal I}_0={\mathcal I} \cap \wedge^\bullet W_0$ is the restriction of ${\mathcal I}$ to the subalgebra $\wedge^\bullet W_0 \subset \wedge^\bullet W$. 

\begin{theorem}[\cite{Zh}]
There is a canonical isomorphism $^\Z\F^\bullet (\Sigma_M) \cong \OS_0^\bullet (M)$ of graded $\Z$-algebras. 
\end{theorem} 

Note that the cohomology groups depend only on the support of a polyhedral fan.
Thus, if two matroids $M_1$ and $M_2$ have Bergman fans with the same support,
the above theorem shows that the two matroids have isomorphic Orlik-Solomon algebras.

The main application for us will be when $M$ is realizable by a hyperplane arrangement in $\C\PP^n$.
Let $Y$ denote the complement of the arrangement.
Then, it is well known (cf., e.g. \cite{OT}) that the projective Orlik-Solomon algebra calculates cohomology of $Y$.
This leads to the following corollary. 

\begin{theorem}\label{OS}
There is a canonical isomorphism $^\Z\F_k (\Sigma_M) \cong H_k (Y; \Z)$. 
\end{theorem}

\subsection{Tropical homology, the cellular version}\label{sec:cellular}
Let $X\subset \T\PP^N$ be a smooth regular projective $\Q$-tropical variety.
The polyhedral decomposition of $X$ into faces gives it a natural cell structure. 

Let $x\in X$ be a point in the relative interior of a face $\Delta_x$ of sedentarity $I$ in $X$. We define $\Sigma(x)$, the {\it fan at} $x$,
to be the cone in $\T^\circ_I \cong \R^{N-|I|}$ consisting of vectors $u\in \T^\circ_I$ such that $x+\epsilon u\in X\cap \T^\circ_I$
for a sufficiently small $\epsilon >0$ (depending on $u$). 

\begin{definition}\label{coefficient_groups} 
We define the coefficient groups $\F_k(x)$ and $\F^k(x)$ to be $^\Z\F_k(\Sigma(x)) \otimes \Q$ and $^\Z\F^k(\Sigma(x)) \otimes \Q$, respectively. 
\end{definition} 

Note that the groups $\F_k(x)$ and $\F_k(y)$ are canonically identified by translation if $x$ and $y$ belong to the relative interior of the same face $\Delta$ of $X$. Thus, we can use the notation $\F_k(\Delta)$. We can also consider the {\em relative coefficient groups}
$\bar\F_k(\Delta)$ defined as $^\Z\F_k(\Sigma_\Delta) \otimes \Q$, 
where $\Sigma_\Delta$ is the relative fan at $\Delta$ 
(see Section \ref{sec:varieties}) if $\Delta$ is mobile; 
if $\Delta$ is sedentary, then $\Sigma_\Delta$ is 
the relative fan at the parent mobile face of $\Delta$ in some affine chart containing the relative interior of $\Delta$. 

If for two points $x,y$ we have $\Delta_x\succ\Delta_y$ then there are
natural homomorphisms
\begin{equation}\label{cosheafmap}
\iota:\F_k(x)\to\F_k(y).
\end{equation}
To define the maps \eqref{cosheafmap} we take an affine chart $U^{(i)}\ni y$. If $I(y)=I(x)$, 
then any face adjacent to $x$ is contained in some face adjacent to $y$ and
the inclusion induces the required map. If $I(y)\neq I(x)$ (note that we must
have $I(y)\supset I(x)$), then the required map
is given by the projection along the divisorial directions indexed by $I(y)\setminus I(x)$. 

For a pair of adjacent faces $\Delta\prec \Delta'$, the map \eqref{cosheafmap} and its dual 
can be rewritten as 
\begin{equation}\label{eq:boundary_map}
\iota:\F_k(\Delta')\to\F_k(\Delta), \qquad \iota^*:\F^k(\Delta)\to\F^k(\Delta').
\end{equation}
This allows us to define a  complex $C_\bullet(X; \F_p)$, where 
$$C_q(X; \F_p)= \oplus \F_p(\Delta).$$
Here, the direct sum is taken over all $q$-dimensional faces of $X$. 
We can write a chain in $C_q(X; \F_p)$ as $\sum \beta_\Delta \Delta$. The boundary map 
$$\dd: C_q(X; \F_p) \to C_{q-1}(X; \F_p)$$
is the usual cellular boundary combined with the maps $\iota$
in \eqref{eq:boundary_map} for any pair of faces $\Delta\succ_1\Delta'$. The groups
$$H_q(X; \F_p) = H_q (C_\bullet(X; \F_p), \dd)$$
are called the {\em {\rm (}cellular{\rm )} tropical $(p, q)$-homology} groups. 

We can consider the dual cochain complex $C^\bullet(X; \F^p)$ of linear functionals on faces $\Delta$ of $X$ 
with values in $\F^p(\Delta)$ and define the differential $\delta$ as the usual coboundary
combined with the maps $\iota^*$ in  \eqref{eq:boundary_map}.
This defines the {\em {\rm (}cellular{\rm )} tropical $(p, q)$-cohomology} groups 
$$
H^{q}(X;\F^p)=H^q(C^\bullet(X; \F^p), \delta).
$$

\subsection{Other homology theories}
We may interpret
$\F_k(x)$ as a system of coefficients suitable to define singular homology groups on $X$.
Namely, we consider finite formal sums 
$$\sum \beta_\sigma \sigma,
$$
where each $\sigma:\Delta^q\to X$ is a singular $q$-simplex which has image in a single face of $X$ and is such that for each relatively open face $\Delta'$
of $\Delta^q$ the image $\sigma(\Delta')$
is contained in a single face of $X$.
We say that $\tau=\sigma|_{\Delta'}$ is a face of $\sigma$. 
Here $\beta_\sigma\in \F_k(\Delta)$, where the relative interior of $\Delta$ contains the image of the relative interior of $\Delta^q$.

These chains form a complex $C^{sing}_\bullet(X; \F_k)$ with the differential $\dd$ given by the standard singular differential combined with the maps $\iota$ in \eqref{eq:boundary_map}.
The elements of $C^{sing}_\bullet(X; \F_k)$ are called {\em tropical chains}. 
The groups
$$H_{p,q}(X)=H_q(C^{sing}_\bullet(X; \F_p), \dd)
$$
are called the {\em singular tropical $(p, q )$-homology}
groups. 

One can also consider \v Cech version of tropical (co)homology thinking of the coefficients $\F_k$ and $\F^k$ as constructible cosheaves and sheaves, respectively, on $X$. We refer the reader for details to \cite{MZh}. 

For the rest of the paper we stick to the cellular version of tropical homology. 

\subsection{Examples of homology computations}\label{sec:examples-calculations}

Here are some examples of calculation of tropical $(p, q)$-homology groups
(these calculations do not use Theorem \ref{thm:main}). 

\begin{example}\label{examples:tropical-homology1} 
Let $X \subset \T\PP^N$ be an $n$-dimensional smooth regular projective $\Q$-tropical variety.
Since $\F_0(\Delta) = \Q$ for any face $\Delta$ of $X$
(and for any pair of adjacent faces $\Delta \prec \Delta'$, the map 
$\iota:\F_0(\Delta')\to\F_0(\Delta)$ is the identity), 
one has
$$H_q(X; \F_0) = H_q(X; \Q)$$  
for any $q = 0$, $\ldots$, $n$.
\end{example} 

\begin{example}\label{examples:tropical-homology2}
Since $\T\PP^1$ is contractible as a topological space, we have
$$H_{0}(\T\PP^1; \F_0) = \Q \quad \mbox{and}\quad H_1(\T\PP^1; \F_0) = 0.$$ 
Consider a subdivision $X$ of $\T\PP^1$ formed by one vertex of sedentarity $0$ (denote this vertex by $O$),
two vertices of positive sedentarity (denote them by $-\infty$ and $+\infty$), and two edges
(denote them $\R_+$ and $\R_-$ in such a way that $\R_\pm$ is adjacent to $\pm\infty$). 
We have $\F_1(O) = \Q$ and $\F_1(-\infty) = \F_1(+\infty) = 0$. 
In addition, $\F_1(\R_+) = \Q$ and $\F_1(\R_-) = \Q$ (the morphisms $\F_1(\R_\pm) \to \F_1(O)$ 
being the identities). 
This gives
$$H_0(\T\PP^1; \F_1)=0 \quad \mbox{and}\quad H_1(\T\PP^1; \F_1) = \Q.$$ 
The above calculation can be easily generalized in order to determine
the tropical $(p, q)$-homology groups of $\T\PP^N$ for arbitrary $N$. 
\end{example} 

\begin{example}\label{examples:tropical-homology3}  
Let $X\subset \T\PP^2$ be a generic $\Q$-tropical line.
The tropical line $X$ has one vertex of sedentarity $0$ (of valency $3$);
denote this vertex by $O$. In addition, $X$ has 
three vertices $V_0$, $V_1$ and $V_2$ of positive sedentarity 
and three edges $E_0$, $E_1$ and $E_2$.
Since $X$ is contractible as a topological space, we have
$$H_{0}(X; \F_0) = \Q \quad \mbox{and}\quad H_1(X; \F_0) = 0.$$ 
Furthermore,  $\F_1(O) = \Q^2$ and $\F_1(V_i) = 0$ for any $i = 0, 1, 2$.
For any edge $E_i$ of $X$, one has $\F_1(E_i) \simeq \Q$, 
and the images of the embeddings of $\F_1(E_i)$ in $\F_1(O) = \Q^2$ 
are the subspaces 
$$
\{(x, y) \in \Q^2 \; | \; y = 0\}, \quad \{(x, y) \in \Q^2 \; | \; x = 0\}, \quad \mbox{and}\quad \{(x, y) \in \Q^2 \; | \; x = y\}. 
$$
This gives
$$H_0(X; \F_1)=0 \quad \mbox{and}\quad H_1(X; \F_1) = \Q.$$ 
The calculation made can be generalized to determine
the tropical $(p, q)$-homology groups 
of any smooth regular projective $\Q$-tropical curve $X \subset \T\PP^N$. 
\end{example} 

Many other examples of calculations of tropical $(p, q)$-homology groups
can be found e.g. in \cite{thesis-Shaw}.


\section{Complex degenerations and tropical limit}
In this section we present a
connection between complex and tropical geometry.
Most of the content presented here develops earlier
results that can be found in e.g. \cite{IMS} and
\cite{Jo}. Our main emphasis is 
Theorem \ref{thm-compactness}
(the compactness theorem)
corresponding
to Proposition 3.9 of \cite{Mik05} in the special
case of hypersurfaces.

Tropical varieties
appear as certain limits of (scaled sequences of) complex varieties. First,
we define the coarse limit as a topological subspace of $\T\PP^N$.
Then, we establish the polyhedrality of this limit and put weights
on its facets in order to obtain  a more refined version.

\subsection{Coarse tropical limit}\label{section:coarse_limit}

\begin{definition}
A {\em scaled sequence} is a set $A$
together with a {\em scaling} map $t:A\to \R$ which is unbounded from above.
A scaled subsequence is a subset $A'\subset A$ with the induced scaling (which is still required to be unbounded).
\end{definition}

We often drop the word ``scaled''. By saying ``$\alpha \in A$ is large'' we mean that $t_\alpha = t(\alpha) \in \R$ is large.
Also, sometimes, we write $\alpha>\alpha'$ instead of $t_\alpha > t_{\alpha'}$. 

\begin{example}\label{eg:scaled_sequence}
Here are some examples of scaled sequences:
\begin{enumerate}
\item conventional sequences $A=\N$ with the inclusion $t:\N\to \R$;
\item $A=\R_{>0}$ with the inclusion $t:\R_{>0} \to \R$;
\item the punctured disc (the most relevant scaled sequence for us)
$$\mathcal D^*= \{ z\in\C \ \suchthat \ 0<|z|<1 \}
$$
with the scaling $t(z)=|z|^{-1}$.
\end{enumerate}
\end{example}

Consider a scaled sequence $X_\alpha\subset\cp^N, \alpha \in A,$ of projective algebraic varieties. From now on we always assume that $t_\alpha  > 1$ for any $\alpha \in A$. We have the map
\begin{equation}\label{log-tpn}
\Log_{t_\alpha}:\cp^N\to\tp^N
\end{equation}
defined by
$(z_0:\dots:z_N)\mapsto (\log_{t_\alpha}|z_0|:\dots:\log_{t_\alpha}|z_N|).$
Note that the map is well defined since the $(N+1)$-tuples of coordinates in $\C\PP^N$ equivalent under multiplication by a nonzero scalar are mapped to $(N+1)$-tuples equivalent
under addition of a scalar in $\T\PP^N$. Also the map respects the coordinate stratifications of $\C\PP^N$ and $\T\PP^N$, that is, each stratum $\cp^N_I\subset \C\PP^N$
is sent to the stratum $\tp^N_I\subset \T\PP^N$.
The set $$\am_\alpha=\Log_{t_\alpha}(X_\alpha)\subset\tp^N$$
is called the {\em amoeba} of $X_\alpha$, cf. \cite{GKZ}.

\begin{defn}[Coarse tropical limit in $\tp^N$]
\label{ctl-tpn}
We say that a closed subset $Y\subset\tp^N$ is the
{\em coarse tropical limit}
of the scaled sequence $X_\alpha\subset\cp^N$
if the amoebas
$\am_\alpha\subset\tp^N$ converge to $Y\subset\tp^N$ in the Hausdorff sense.
\end{defn}

This means that if we choose a metric $d$ compatible with the topology on $\tp^N$,
then the Hausdorff distance
$$
\max\{\sup\limits_{x\in  \am_\alpha}d(x,Y), \sup\limits_{y\in Y} d(\am_\alpha,y)\}
$$
between $\am_\alpha$ and $Y$ tends to 0 as $t_\alpha \to \infty$. Note that if the Hausdorff distance between two closed subsets $Y, Y' \subset \T\PP^N$ is 0 then $Y=Y'$. That is, if a limit exists, it is unique.

\begin{example}\label{example:points}
Since $\T\PP^N$ is compact, any sequence of points $c_\alpha \in \C\PP^N$ has a subsequence which has a course tropical limit.
\end{example}

\begin{definition}
We say that a polyhedral complex $Z \subset \T\PP^N$  of dimension at most $n$ is an {\em attractor} for a sequence of $n$-dimensional varieties $X_\alpha \subset \C\PP^N$ if it contains all accumulation points of the amoebas $\am_\alpha$. 
\end{definition}

\subsection{Weights on facets of an attractor}\label{section_weights} 
Let $Z\subset\tp^N$ be an attractor for a sequence of $n$-dimensional varieties $X_\alpha\subset \C\PP^N$ of degree $d$. Denote by $Z^\epsilon$ a small tubular $\epsilon$-neighborhood of $Z$. We assume that $\alpha$ is sufficiently large, so that the amoebas $\am_\alpha$ sit entirely in $Z^\epsilon$.

Fix a subset $I \subset \{0, 1, \ldots, N\}$, and
let $M$ be an $(N-n)$-dimensional cooriented $I$-ball in $ \T\PP^N$ such that $\dd M \cap Z^\epsilon=\emptyset$ (recall that $\dd M$ means the closure in $\T\PP^N$ of the mobile part of the boundary of $M$). 
We define $\mathcal K^\alpha_{M} \in H_n((\C^\times)^{N-I}\times \C^I;\Z)=\Lambda^n(\Z^N/\Z^I)$ to be the intersection class of a small perturbation of $\Log_{t_\alpha}^{-1}(M)$ with $X_\alpha \cap ((\C^\times)^{N-I}\times \C^I)$. 

\begin{lemma}\label{lemma:linking}
Let $M_s$, $s \in [0, 1]$, be a continuous family of 
$(N-n)$-dimensional $I$-balls such that $M_s \cap Z^\epsilon=\emptyset$ for each $s \in [0, 1]$.  
Then, $\mathcal K^\alpha_{M_0} = \mathcal K^\alpha_{M_1}$. 
\qed
\end{lemma}

We now introduce two special types of $(N-n)$-dimensional $I$-balls in $\T \PP^N$. 
First, let $F$ be an oriented $n$-dimensional facet of $Z$ with the relative interior $F^\circ \subset\tp^\circ_I$.
 A {\em membrane} $M_F^\epsilon \subset \T\PP^N$ is a small $(N-n)$-dimensional $I$-ball, such that $\dd M_F^\epsilon \cap Z^\epsilon=\emptyset$ and $M_F^\epsilon$ intersects $Z$ in a single point $x \in F^\circ$. An orientation of $F$ induces a coorientation of $M_F^\epsilon$.
This defines $\mathcal K^\alpha_{M_F^\epsilon} \in H_n((\C^\times)^{N-I}\times \C^I;\Z)=\Lambda^n(\Z^N/\Z^I)$.
Membranes through $F^\circ$ allow us to associate to $F$ a class $\mathcal K^\alpha_F  \in \Lambda^n (\Z^N/\Z^I)$ (cf. Lemma \ref{lemma:linking}). 

Another type comes from choosing a cooriented $(N-n-|I|)$-dimensional rational 
linear subspace $L_I$ in $\tp^\circ_I \cong \R^{N-I}$.
Let $L$ denote the pullback of $L_I$ to $\T\PP^N$ via the projection $\pi_I^\T : \T\PP^N \to \T\PP_I$. 
An $(N-n)$-dimensional $I$-ball $M_{L_I}$ parallel to $L$ with $\dd M_{L_I} \cap Z^\epsilon=\emptyset$ 
is called an {\em $L_I$-crepe}, or simply, a {\em crepe}. 

Let $y\in Z_I$ be a point such that its $\epsilon$-neighborhood in $\T\PP^N$ is disjoint from all $Z_J$ for 
$J \subsetneq I$, and let $L_I$ be transversal (in $\T\PP^\circ_I$) to each $n$-dimensional face $F$ of $Z_I$ such that $y$
belongs to the closure of $F$. 
Then, through any point in the $\epsilon$-neighborhood of $y$ one can trace an $L_I$-crepe. 

Now let $M_{L_I}$ be an $L_I$-crepe, and let $\LL$ be the product of $\C^I$ with the $(N-n-|I|)$-dimensional subgroup 
of $(\C^\times)^{N-I}$ corresponding to $L_I$ in $\R^{N-I}$. Then, to any point $z\in \Log^{-1}_{t_\alpha}(M_{L_I})$
we can associate a complex $(N-n)$-dimensional submanifold $\LL_z \subset (\C^\times)^{N-I}\times \C^I$, 
which is the intersection of $\Log_{t_\alpha}^{-1}(M_{L_I})$
with the translate $\LL'$ of  $\LL$ such that $\LL'$ contains the point $z$.   

The following calculation is a direct consequence of the Poincar\'e duality.

\begin{lemma} \label{lemma:LL}
Let $L_I\subset\R^{N-I}$ be a cooriented $(N-n-|I|)$-dimensional 
rational linear subspace, and let $M_{L_I}$ be an $L_I$-crepe. Then,
the intersection number of $X_\alpha$ with $\LL_z$ 
equals $\mathcal K^\alpha_{M_{L_I}} \wedge \Vol_{L_I} / \Vol_{\Z^{N-I}}$ for any $z\in \Log_{t_\alpha}^{-1}(M_{L_I})$.
\qed
\end{lemma} 

A very useful class of $I$-balls is provided by the intersection of the above two types.
Let $F$ be an oriented $n$-dimensional facet of $Z$ 
with the relative interior $F^\circ \subset\tp^\circ_I\cong \R^{N-I}$, 
and let $L_I$ be an $(N-n-|I|)$-dimensional 
rational linear subspace in $\R^{N-I}$. If $L_I$ is integrally transversal to $F$, 
we call a small crepe $M^\epsilon_{L_I,x}$ through a point $x\in F^\circ$ a {\em linear membrane}. 
By Lemma \ref{lemma:linking} any linear membrane through $F^\circ$ calculates the class $\mathcal K^\alpha_F$. 

\begin{lemma}\label{lem-weight}
There exists $T_F \in \R$ such that for any $\alpha$
with $t_\alpha >T_F$  the class $\mathcal K^\alpha_F$
is a non-negative integer multiple of $\Vol_F$. 
\end{lemma} 

\begin{proof}
We consider the case 
$F^\circ\subset\R^N$. The sedentary case is similar.

Let us fix a basis $e_1,\dots ,e_N\subset \Z^N$ such that $\Vol_F=e_1\wedge\dots \wedge e_n$. 
For every map of sets 
$$h: \{n+1, \dots, N\} \to \{0,1,\dots, n\}, 
$$ 
we consider the subspace $L(h)$ spanned by the vectors $e_{n+1}+e_{h(n+1)}, \dots, e_N+e_{h(N)}$ 
oriented such that  $\Vol_{L(h)}=(e_{n+1}+e_{h(n+1)}) \wedge \dots \wedge (e_N+e_{h(N)})$. 
We set $e_0=0$; in particular, $\Vol_{L(0)}=e_{n+1}\wedge\dots \wedge e_N$. 
Note that $\Vol_F \wedge \Vol_{L(h)}=1$ for any $h$ (here we tacitly divide by the volume element $e_1\wedge\dots \wedge e_N$), that is, all $L(h)$ are integrally transversal to $F$. 
It is easy to see that the collection of all $(n+1)^{N-n}$ primitive polyvectors $\{\Vol_{L(h)}\}$ generate  $\Lambda^{N-n} (\Z^N)$. 

We assume that $\epsilon$ is small enough so that 
every $L(h)$ is parallel to a linear membrane through 
some point of $F^\circ$. We write the class 
$$\mathcal K^\alpha_F = \sum_{|J|=n} c^\alpha_J e_J
$$
in the basis of primitive polyvectors $e_J=\wedge_{i\in J} e_J \in \Lambda^n(\Z^N)$,
where $J$ runs over the increasing length $n$ sequences  in $\{1,\dots,N\}$. 

Since the polyvectors $\{\Vol_{L(h)}\}$ generate  $\Lambda^{N-n} (\Z^N)$, the coefficients $c^\alpha_J$
can be calculated by taking products $\mathcal K^\alpha_F \wedge \Vol_{L(h)}$, that is, according to Lemma  \ref{lemma:LL}, by locally intersecting $X_\alpha$ with complex manifolds $\LL(h)$.  
Thus, all $c^\alpha_J$ are uniformly (for all large $\alpha$) bounded by some constant
(which depends
on the degree of varieties $X_\alpha$ and the basis $e_1,\dots ,e_N$). 

Now let $w\in\Lambda^{N-n}(\Z^N)$ be any primitive polyvector
such that $\Vol_F  \wedge w =1$. We assume that $\epsilon$ 
is small enough so that there is a linear membrane parallel to $w$ and passing through some point of $F^\circ$. 
Then, since complex manifolds intersect non-negatively, we must have $\mathcal K^\alpha_F  \wedge w \ge 0$. In particular, wedging $\mathcal K^\alpha_F$ with $w=\Vol_{L(0)}$ gives $c^\alpha_{J_0} \ge 0$, where $J_0=\{1,\dots,n\}$. We will show that $c^\alpha_J=0$ for all  $J\ne J_0$. 

Given $J$, consider two disjoint ordered subsets of $\{1,\dots, N\}$: 
$$\{i_1,\dots ,i_k\}= J \setminus J_0, \quad
\{j_1,\dots, j_k\}= J_0 \setminus J, 
$$ 
and the following primitive polyvector
$$w_J= \pm \wedge_{\ell=1}^k (e_{i_\ell} \pm C e_{j_\ell})\wedge e_{\{n+1,\dots, N\}\setminus J} \in \Lambda^{N-n}(\Z^N),
$$ 
where $C$ is an integer. 
Then, $\Vol_F \wedge w_J=1$ (with the right choice of the $\pm$ sign in front). On the other hand, 
$$\mathcal K^\alpha_F \wedge w_J = \pm C^k c^\alpha_J + \sum_{J' \ne J} \pm C^{<k} c^\alpha_{J'}. 
$$
By appropriate choice of the $\pm$ signs and large enough $C$ this can be made negative unless $c^\alpha_J=0$ 
(recall that all $c^\alpha_J$ are uniformly bounded). 
\end{proof} 

\begin{defn}
For $\alpha$ with $t_\alpha >T_F$, the number $w_\alpha(F):= \frac{\mathcal K_{F}^\alpha}{\Vol_F}\in \Z_{\ge 0}$ is called {\em the weight} of the facet $F$.
\end{defn}

When we write  $w_\alpha(F)$ we assume that $t_\alpha > T_F$. The weights $w_\alpha(F)$ have {\it a priori} bounds in terms of the degree of $X_\alpha$.
One also observes that $w_\alpha(F)$ does not depend on the choice of orientation of $F$ because the orientation affects signs of both $\mathcal K_{F}^\alpha$ and $\Vol_F$. If $Z'  \subset \T\PP^N$ is another attractor, and $F$ is a facet of both, then for large $\alpha$ the weights $w_\alpha(F)$, considering $F$ as a face of either $Z$ or $Z'$, are the same.

\begin{proposition}\label{WY}
Let $Z\subset\tp^N$ be an attractor for a sequence of $n$-dimensional varieties $X_\alpha\subset \C\PP^N$ of degree $d$. 

\begin{enumerate}
\item Let $F$ be an $n$-dimensional facet of $Z$ with weights $w_\alpha(F)$.
If for any $T \in \R$ there exists an index $\alpha$ with $t_\alpha >T$ and $w_\alpha(F)>0$, then all points of $F$ are accumulation points
of the sequence $X_\alpha$. 

\item Conversely, if $x\in Z$ is an accumulation point of the sequence of $X_\alpha$,
then for any $T \in \R$ there exists an index $\alpha$ with $t_\alpha >T$ and an $n$-facet $F$ of $Z$
such that $x\in F$ and $w_\alpha(F)>0$.
\end{enumerate} 
\end{proposition}

\begin{proof}
(1) If $w_\alpha(F) >0$, then $\mathcal K^\alpha_F \neq 0$. 
This means that any membrane through a point in $F^\circ$ intersects $\am_\alpha$. 
Thus, any point of $F$ is an accumulation point of the sequence $X_\alpha$. 

(2) Let $x \in Z_I$ be an accumulation point. First, assume $x\in \R^N$. Passing to a subsequence, we pick $z_\alpha \in X_\alpha$ 
such that the sequence $x_\alpha = \Log_{t_\alpha} (z_\alpha)$ 
converges to $x$. 

Choose a small crepe $M_{L,x}$ passing through $x$ 
and consider $L$-crepes through $x_\alpha$ which are small translates of $M_{L,x}$. All these translates define classes $\mathcal K^\alpha_{M_{L,x}}$ (independent of $x_\alpha$).  By using positivity of intersection of complex varieties $X_\alpha$ and  $\LL_{z_\alpha}$ we see that all classes $\mathcal K^\alpha_{M_{L,x}}$ are non-zero. 

Perturbing $M_{L,x}$ we can move off all points of its intersection
with $Z$ into the interiors of facets $F_j$ adjacent to $x$,
so that we have 
$$\mathcal K^\alpha_{M_{L,x}} =\sum\limits_j \mathcal K^\alpha_{F_j},
$$
and thus at least one of the facets $F_j$ must have nonzero weight.

We modify $Z$ by removing all faces that are not contained in $n$-facets with non-zero weight and then 
proceed by induction on $I \subset \{0, \dots, N\}$ by applying the above argument to the modified attractor.
\end{proof}

\subsection{Tropical limit for hypersurfaces}\label{sec:trop_hypersurfaces}
Tropical polynomials are the analogs of the classical polynomials where the addition $x+y$ is replaced by $\max\{x,y\}$ and the multiplication $xy$ is replaced by the sum $x+y$.
For further analogies and details see, e.g \cite{Mik06}.
A useful observation in this ``tropical arithmetics'' is the following inequality:
\begin{equation}\label{trc-estimate}
\max_{j \in S} \{\ell_j\} \le \log_t \sum_{j \in S} t^{\ell_j} \le \max_{j \in S} \{\ell_j\} + \log_t |S|
\end{equation}
for any finite set $S$, any tropical numbers $\ell_j \in \T$ (where $j \in S$),  and any $t > 1$. 

Fix the dimension $N$, and let
$$ \Delta^\Z_d := \{m\in \Z^{N+1} \suchthat m_0+\dots+m_N=d \;\;\; \text{\rm and} \;\;\; m_i\ge 0\} 
$$
be the set of integral points of the $N$-dimensional simplex $\Delta_d$ of size $d$ in $\R^{N+1} \supset\Z^{N+1}$.
For any function $a: {\Delta^\Z_d} \to \T$, not identically equal to $-\infty$,
we define the {\it degree $d$ homogeneous tropical polynomial}
$P_a: \T^{N+1} \to \T$ as the Legendre transform of the function $-a$:
$$P_a:=\max_{m\in \Delta^\Z_d} \{mx+a(m)\}.
$$
Note that since all components of $m$ are non-negative,
the expressions $mx+a(m)$ do make sense for all $x\in  \T^{N+1}$. 

The restriction to any open stratum $\T^\circ_I \subset \T^{N+1}$ of a tropical polynomial $P_a$
is a convex  piecewise linear function $P_{a,I}$.
The set of points $x\in\T^{N+1} = \bigcup \T^\circ_I$ such that $P_{a, I}$ is not smooth at $x$ or equal to $-\infty$
(where $I$ is the refined sedentarity of $x$) 
is an $N$-dimensional polyhedral complex in $\T^{N+1}$ invariant under the diagonal translation by $\R$.
Thus, it descends to an $(N-1)$-dimensional polyhedral complex $V_a$ in $\T\PP^N$. 
The finite part $V^\circ_a \subset \R^N$ of this complex is dual to 
a subdivision of the Newton polytope $\Delta_a$
of $P_a$; this subdivision is given by the upper convex hull of the graph of $-a$. 

We assign the weights to the facets of $V_a$
as follows. 
On the facets of the closure of $V^\circ_a$ in $\T\PP^N$ 
the weights are given by the integral length of the gradient change of $P_a$, 
which is the same as the integral length of the dual edge in the subdivision (induced by $-a$) 
of the Newton polytope $\Delta_a$.
The weight on each boundary divisor $\T\PP_{\{i\}}$ is given by the integral distance 
from the Newton polytope $\Delta_a$
to the facet $m_i=0$ in the simplex $\Delta_d$. 
We discard from $V_a$ the boundary components of weight $0$. 

\begin{prop}
The hypersurface $V_a$ associated to a homogeneous tropical polynomial $P_a:\T^{N+1}\to\T$
is a weighted balanced polyhedral complex of dimension $N-1$ in $\tp^N$.
Conversely, any $(N-1)$-dimensional
weighted balanced polyhedral complex $Y\subset\tp^N$ may be presented
as $V_a$ for some homogeneous polynomial $P_a:\T^{N+1}\to\T$.
\end{prop}
\begin{proof}
The statement that $V_a$ is a polyhedral complex is clear from the discussion above.
Namely, the intersection of $V_a$ with any $\T\PP^\circ_{I}$ is of one of the following three types:  
\begin{enumerate}
\item the hypersurface defined by the function $a$ restricted to the $I$-th face of $\Delta^\Z_d$,
\item or empty, if $a$ has only one finite (not $-\infty$) value on the $I$-th face of  $\Delta^\Z_d$,
\item or the entire stratum $\T\PP^\circ_{I}$, if all values of $a$ are $-\infty$ on the $I$-th face of $\Delta^\Z_d$. 
\end{enumerate} 
The balancing condition needs to be checked only for $V_a^\circ$,
where it is a straightforward property of the Legendre transform
(cf., e.g., Propositions 2.2 and 2.4 of \cite{Mik-pp},
or Theorem 3.15 of \cite{Mik05}). 

Conversely, given a weighted balanced $(N-1)$-dimensional polyhedral complex
$Y\subset\tp^N$, consider the cylinder, invariant under the diagonal translations, over $Y^\circ:=Y\cap\R^N\subset\R^N$ in $\T^{N+1}$.
This is a balanced polyhedral complex of codimension $1$.
Hence, by Propositions 2.2 and 2.4 of \cite{Mik-pp}  quoted above,
it is given as the corner locus of a homogeneous tropical polynomial $P_{a^\circ}$ defined uniquely up to an affine linear function. 
The closure $\overline{Y^\circ}$ of $Y^\circ$ in $\T\PP^N$ is a polyhedral complex by Lemma \ref{lemma:closure}. 
Then, by parallel shifting of the support of $a^\circ$, it is easy to achieve that 
this support lies in the positive octant and its integral distances 
to the coordinate hyperplanes are the given weights 
of the corresponding boundary divisors $\T\PP_{\{i\}}$ in $Y$.
\end{proof} 

\begin{definition}\label{defn:trophyper_degree}
The {\em degree} of an $(N-1)$-dimensional weighted balanced polyhedral complex $Y\subset\tp^N$ 
is the degree of a tropical polynomial $P_a$ such that $V_a$ coincides with $Y$. 
\end{definition} 

An important observation is that if we add a constant to the function $a:\Delta^\Z_d \to \T$ the tropical polynomial $P_a$ changes
by adding (the same) constant. That is, both $V_a$ and the Newton polytope $\Delta_{a}$
remain  the same. Hence, the tropical hypersurface $V_a$ is well defined for $a\in \T\PP^{\Delta^\Z_d}$. 

Also given $P_a$, there is still some freedom in choosing the function $a:\Delta^\Z_d \to \T$
as long as the upper convex hull of the graph of $-a$ induces the same subdivision of the Newton polytope $\Delta_a$.
One can always take the ``minimal'' representative of $a$ by setting $a(m)=-\infty$
for all non vertices of the induced subdivision of $\Delta_a$.
This reduces the ambiguity in $P_{a}$ to an additive constant, cf. Theorem 3.15 of \cite{Mik05}.

Consider a scaled sequence of complex hypersurfaces $X_\alpha\subset\cp^N$ of degree $d$.
The main result of this subsection is that the sequence $X_\alpha$ has a coarsely convergent subsequence 
and its tropical limit $Y\subset\tp^N$ is a polyhedral complex. 

For an interior point $x\in \R^N \subset \T\PP^N$ we consider its preimage $T_x\in  \C^{N+1}\setminus \{0\}$ under the composition of the two maps:
$$ \C^{N+1}\setminus \{0\} \rightarrow \C\PP^N \rightarrow \T\PP^N.
$$
The first map is the quotient by $\C^*$, and the second is the ${\Log_{t_\alpha}}$.

\begin{definition}[\cite{FPT}]
For a hypersurface $X_\alpha\subset\cp^N$ let $f_\alpha$ be its defining polynomial.
Let $x\in \R^N \subset \T\PP^N$ be a point outside the amoeba $\am_\alpha$. We define the  functional $\ind_\alpha (x)$ on the space of loops in $T_x$ as follows. For a loop $\gamma \subset T_x$ we set
\begin{equation}
\ind_\alpha (x) : \gamma \mapsto \frac1{2\pi i} \int_\gamma d\log f_\alpha.
\end{equation}
\end{definition}

\begin{remark}
One may think of the value $\ind_\alpha(x)$ on the loop $\gamma$ as the linking number of $\gamma$ and the affine cone $\hat X _\alpha$ over $X_\alpha$ in $\C^{N+1}$. If $\mu\subset\cp^N$ is a holomorphic disc with boundary $\dd \mu= \gamma$,
then the value of $\ind_\alpha (x)$ on $\gamma$ is the intersection number of $\mu$ with $\hat X_\alpha$.
\end{remark}

There are a couple of immediate observations: $\ind_\alpha (x)$ depends only on the homology class of the loop, and it is a locally constant function of $x$. For convenience, we extend the map $\ind_\alpha$ to the boundary strata in $\T\PP^N$ by continuity.
Note that $T_x$ is isomorphic to $\C^* \times (S^1)^N$,
and for any $x\subset \R^N$ the cohomology $H^1(T_x; \Z)\cong \Z^{N+1}$ 
can be naturally identified with the lattice of the Laurent monomials in the coordinates on $\C^{N+1}$.
Thus, $\ind_\alpha$ maps connected components of $\T\PP^N \setminus \am_\alpha$ to this lattice. 

\begin{proposition}[\cite{FPT}]\label{prop:fpt-injective}
Let $X_\alpha\subset \C\PP^N$ be a degree $d$ hypersurface. The map $\ind_\alpha$ is injective on the components of $\T\PP^N \setminus \am_\alpha$ and its image lies in $\Delta^\Z_d \subset \Z^{N+1}$.
\end{proposition}

Here is the dictionary between the weights $w_\alpha(F)$ and the indices $\ind_\alpha$ for hypersurfaces.

\begin{proposition}\label{prop:index-weight}
Let $X_\alpha \subset \C\PP^N$ be a scaled sequence of hypersurfaces which coarsely converges to $Y\subset \T\PP^N$.

\begin{itemize}
\item If $F\subset Y$ is a facet separating two components $\sigma,\sigma'$ of $\T\PP^N \setminus Y$, 
then the weight $w_\alpha(F)$ is equal to the integral length of the vector $\ind_\alpha(\sigma) - \ind_\alpha(\sigma')$. 

\item If $F$ is a coordinate hyperplane $\T\PP_{\{j\}}\subset Y$, then $w_\alpha(F)$ is equal to $(\ind_\alpha)_j(\sigma)$, 
the $j$-th component  of $\ind_\alpha (\sigma)\in \Delta^\Z_d$, where $\sigma$ is an adjacent component of $\T\PP^N \setminus Y$.
\end{itemize}
\end{proposition} 

\begin{proof}
Let $F\subset Y$ be a facet with $F^\circ\subset\R^N$
which separates two components $\sigma,\sigma'$ of $\T\PP^N \setminus Y$. 
Let $L \subset \R^N$ be a line which intersects $F$ integrally transversally  at a point $x \in F^\circ$, and
let $M_{F,L} \subset L$ be a small interval around $x$. 
Denote by $\hat F \subset \R^{N+1}$ the corresponding facet of the cylinder over $Y$. 
Let $\hat L$ be a line integrally transversal to $\hat F$ which projects to $L$, 
and let $\hat M_{F,L} \subset \hat L$ be the corresponding lift of the interval $M_{F,L}$. 

Then, by Lemma \ref{lemma:LL}, the weight $w_\alpha(F)$ can be computed 
as the local intersection number of $X_\alpha$ and the complex subgroup $\mathcal L \cong \C^* \subset (\C^*)^N$ 
corresponding to $L$. This equals the number of zeros of the defining polynomial $f_\alpha$ in the annulus $\hat{\mathcal L} \cap \Log_{t_\alpha}^{-1} (\hat M_{F,L}) \subset \hat {\mathcal L} \cong \C^*$, bounded by two loops $\gamma_1, \gamma_2$ in the homology class $\<\hat L\> \in \Z^{N+1}$. Thus, by Cauchy's formula this number is
$$ \frac1{2\pi i} \int_{\gamma_1 \cup (- \gamma_2)} d \log f_\alpha,
$$
which also can be computed as the value of  $\ind_\alpha (x_1)  - \ind_\alpha (x_2)$ on $\<\hat L\>$, where $x_1 = \Log_{t_\alpha}(\gamma_1), x _2 = \Log_{t_\alpha}(\gamma_2)$.
Since $\hat L$ intersects $\hat F$ integrally transversally, this is precisely the length of the vector $\ind_\alpha(\sigma) - \ind_\alpha(\sigma')$. 

For the coordinate hyperplane $F=\tp_{\{j\}}\subset\tp^N$,
let
$x=(x_0 : \ldots : x_N)$ be a point in a component of $\T\PP^N\setminus Y$ adjacent to it. 
Then, for $\alpha$ sufficiently large, the value of $(\ind_\alpha)_j(x)$ 
can be calculated by Cauchy's formula as the intersection number of the affine cone $\hat X_\alpha$ over $X_\alpha$ in $\C^{N+1}$ with the disk
$$\{(z_0,\dots,z_N)\in\C^{N+1}\ \suchthat z_i=t_\alpha^{x_i},\ i\ne j,\quad |z_j| < t_\alpha^{x_j}\}.
$$
On the other hand, by Lemma \ref{lemma:LL} the weight $w_\alpha(F)$ can be calculated as the intersection number of $X_\alpha$ with the disk
$$\{(z_0 : \ldots : z_N) \in\C\PP^N \suchthat z_i=t_\alpha^{x_i},\ i\ne j,\quad |z_j| < t_\alpha^{x_j}\}. 
$$ 
These two calculations clearly agree. 
\end{proof}

\begin{definition}\label{fine-conv-hyp}
We say that a scaled sequence of hypersurfaces
$X_\alpha \subset \C\PP^N$
{\em tropically converges} to a closed set $Y\subset \T\PP^N$
if the following conditions hold:
\begin{itemize}
\item
$Y$ is the coarse tropical limit of $X_\alpha$;
\item
for every
$x\in\tp^N\setminus Y$ the sequence $\ind_\alpha(x)\in\Z^{N+1}$ is eventually constant
(that is, constant for sufficiently large $\alpha$).
\end{itemize}
\end{definition}

By Proposition \ref{prop:fpt-injective} the degrees of $X_\alpha$ in a tropically convergent sequence finally stabilize. The second condition makes sense since a point $x\in \tp^N\setminus Y$ misses
all amoebas $\am_\alpha$ for sufficiently large $\alpha\in A$. In this case we define
$\ind(x)=\ind_\alpha(x)$,
where $\alpha$ is taken to be sufficiently large.
It is clear that the function $\ind: \T\PP^N\setminus Y\to\Z^{N+1}$ is locally constant. 

\begin{lemma}
Let $Y$ be the coarse tropical limit of a scaled sequence
of hypersurfaces $X_\alpha\subset\C\PP^N$ of degree $d$. Then, 
there is a scaled subsequence $A'\subset A$ such that $X_\alpha$, $\alpha\in A'$, tropically converges to $Y$.
\end{lemma}

\begin{proof}
The set $\ind(\T\PP^N\setminus Y)\subset \Delta^\Z_d$ is finite.
To get a converging subsequence we
pass to a subsequence of constant $\ind_\alpha$ for each connected component of $\T\PP^N\setminus Y$.
\end{proof}

There is another natural way to pass to a tropical limit of the sequence of complex hypersurfaces $X_\alpha\subset\cp^N$ of degree $d$. The coefficients of the defining polynomials for $X_\alpha$ generate the sequence of points $c_\alpha \in \C\PP^{\Delta^\Z_d}$. Its coarse tropical limit $c\in \T \PP ^N$ (which always exists after passing to a subsequence, cf. Example \ref{example:points})
defines a tropical hypersurface $V_c\subset\tp^N$.

\begin{proposition}\label{thm-hypreform}
Let $X_\alpha\subset\cp^N$ be a scaled sequence of hypersurfaces of degree $d$.

If $c\in\tp^{\Delta^\Z_d}$ is the coarse tropical limit of the corresponding sequence of coefficients $c_\alpha \in \C\PP^{\Delta^\Z_d}$,
then the sequence $X_\alpha$ tropically converges to {\rm (}the support of{\rm )\/} $V_c\subset\tp^N$
in the sense of Definition \ref{fine-conv-hyp}.

Conversely, if $Y\subset\tp^N$ is the coarse tropical limit of the sequence  $X_\alpha$,
then $Y$ is the support of a tropical hypersurface $V_c$ of degree $d$, where $c\in \tp^{\Delta^\Z_d}$
is an accumulation point of the sequence $c_\alpha \in \C\PP^{\Delta^\Z_d}$.
\end{proposition}

\begin{proof}
Suppose that $x\in\tp^N\setminus V_c$. Then, there exists $m\in\Delta^\Z_d$
such that $mx+c(m) > m'x+c(m')$ for all $m'\ne m$. Then,
$| t_\alpha^{mx+c(m)}|>|\sum\limits_{m'\neq m} t_\alpha^{m'x+c(m')}|$ for sufficiently large $t_\alpha$, which, in turn, by \eqref{trc-estimate} means
\begin{equation}\label{c-est}
|c^{\alpha}_m z^m| > |\sum_{m'\neq m} c^{\alpha}_{m'}z^{m'}|
\end{equation}
 for all $z\in \Log_{t_\alpha}^{-1}(x)$.
This implies that $\ind_\alpha(x)=m$ for all sufficiently large $t_\alpha$ (cf. \cite{FPT}).

Furthermore, \eqref{c-est} holds simultaneously for all $x\in\tp^N$ outside
of an $\epsilon$-neighborhood of $V_c$ in $\tp^N$ by \eqref{trc-estimate} (with different values of $m$ for different components of $\T\PP^N \setminus V_c$).
Thus $\lim_{\alpha\in A} (\sup_{x\in\am_\alpha}d(x,V_c))=0.$

Now we want to show that $\lim_{\alpha \in A} (\sup_{y\in V_c}d(\am_\alpha,y))=0$.
By compactness of $V_c$,
it suffices to see that for each open set $U\subset\tp^N$
with $U\cap V_c\neq\emptyset$
there exists $t' >0$ such that $\am_{\alpha}\cap U\neq\emptyset$
for each $\alpha$ with $t_{\alpha} > t'$. Moreover, it is sufficient to prove this for
such $U$ that $U\cap\R^N$ is convex.
Since $V_c$ is a tropical hypersurface, such $U$ must intersect
at least two components of $\tp^N\setminus V_c$ with different values
of $\ind$. Let $y_1,y_2\in U$ be two points with $\ind(y_1)\neq\ind(y_2)$.
The interval connecting $y_1$ and $y_2$ must intersect $\am_\alpha$ for large $\alpha$.
Therefore, we conclude that $X_\alpha$ tropically converges to $V_c$.

Conversely, suppose that $Y\subset\tp^N$ is the coarse tropical limit of the sequence  $X_\alpha$.
Consider an accumulation point $c$ of the coefficient sequence $c_{\alpha}\in\tp^{\Delta^\Z_d}$
(which must exist by compactness of $\tp^{\Delta^\Z_d}$).
Let $A'$ be a scaled subsequence of $A$ such that
$\lim\limits_{\alpha\in A'}c_{\alpha}=c$.
We have already proved that $X_\alpha$ must converge to $V_c$ for the subsequence $A'$.
Hence $V_c=Y$.
\end{proof}

\begin{corollary}[Compactness theorem for hypersurfaces]\label{coro-hyp}
Let $X_\alpha\subset\cp^N$ be a scaled sequence of hypersurfaces of degree $d$.
Then, $X_\alpha$ has 
a subsequence which tropically converges to a tropical hypersurface $Y\subset \T\PP^N$ of degree $d$.
\end{corollary}
\begin{proof}
Choose an accumulation point of the coefficient sequence $c_{\alpha}\in\tp^{\Delta^\Z_d}$ and 
pass to a subsequence converging to this accumulation point.
By Proposition \ref{thm-hypreform} the tropical limit $Y$ of the hypersurfaces 
$X_\alpha$ is represented as a tropical hypersurface of degree $d$.
\end{proof}

\subsection{Compactness theorem}

In this subsection, we prove that a scaled sequence of $n$-dimensi\-onal varieties of universally bounded degrees in $\C\PP^N$
has a scaled subsequence which tropically converges to a balanced weighted $n$-dimensional polyhedral complex in $\tp^N$. 

\begin{lemma}\label{lem-Z}
Let $X_\alpha\subset\cp^N$ be a scaled sequence of $n$-dimensional varieties of degree $d$.
Then, $X_\alpha$ has a subsequence which possesses an attractor $Z\subset\tp^N$. 
\end{lemma} 

\begin{proof}
We realize $Z$ as the intersection of several tropical hypersurfaces.
In the line of proof we may need to pass to subsequences. 

Let $\cp_I\subset\cp^N$ be a coordinate projective subspace,
and let $\pi_I:\cp^N\to\cp_{I}$ be the projection from the dual coordinate projective subspace
$\cp_{\hat I}$ (where ${\hat I} = \{0, \ldots N\} \setminus I$). 
It is a rational map not defined at $\cp_{\hat I}$. For a closed subset $X\subset \C\PP^N$ we denote by $\pi_I (X)$ its image, that is the closure of  $\pi_I (X\setminus \cp_{\hat I})$ in $\cp_I$. Similarly, we can consider $\pi^\T_I:\tp^N\to\tp_{I}$, the tropical version of the projection.

Define $\mathcal I$ to be the set of subsets $I\subset \{0,\dots, N\}$ such that
$\pi_I(X_\alpha)\subset\cp_{I}$ are hypersurfaces for all large $\alpha$.
Generically $\mathcal I$ consists of the subsets with $N-n-1$ elements. 
However, it may happen that for some $I$ with  $|I|=N-n-1$ the images $\pi_I(X_\alpha)\subset\cp_{I}$ 
are of codimension higher than 1 for all large $\alpha$. 
Then, we will need to pass to a further projection to $\cp_{I'}$ with $I'\supset I$
so that $\pi_{I'}(X_\alpha)$ are hypersurfaces in $\cp_{I'}$ (after passing to a subsequence) for all large $\alpha$. 

For each $I\in \mathcal I$, the degrees of $\pi_I(X_\alpha)\subset\cp_{I}$ 
are bounded by $d=\deg (X_\alpha)$, 
and hence so are the degrees of the cones $\pi_I^{-1}(\pi_I(X_\alpha))$. 
Thus, after passing to a subsequence we may assume that for all $I\in \mathcal I$ 
the degrees $d_I$ of the hypersurfaces $\pi_I^{-1}(\pi_I(X_\alpha))$ are fixed. 
For each $I\in \mathcal I$, by Corollary \ref{coro-hyp}, there exists a tropical limit $Z_I$ 
of $\pi_I^{-1}(\pi_I(X_\alpha))$ 
which is a tropical hypersurface 
of degree $d_I$. 

Now observe that $X_\alpha \subset \pi_I^{-1}(\pi_I(X_\alpha))$, hence
$$ \Log_{t_\alpha} (X_\alpha) \subset  \Log_{t_\alpha} (\pi_I^{-1}(\pi_I(X_\alpha))= (\pi_I^\T)^{-1} ( \Log_{t_\alpha} (\pi_I(X_\alpha)).
$$
This means that $Y$ is contained in any $Z_I$ for $I\in \mathcal I$. We define
$$Z:= \bigcap_{I\in \mathcal I} Z_I.
$$
Then $Z\supset Y$, and  we claim that $Z$ is the union of polyhedral complexes in $\T\PP^N$ 
of dimension not greater than $n$. For this note that each hypersurface $Z_I$ is a polyhedral complex. 
The intersection of the hypersurfaces restricted to some stratum $\tp_{I'}^\circ$
is given by the intersection of the restrictions of those hypersurfaces to $\tp_{I'}^\circ$. 
That shows the face incidence property 
of Definition \ref{defn:polycomplex}. 

Finally we show that $Z$ cannot contain faces of dimension greater than $n$.
Indeed, suppose $F$ is such a face. Then, there is a subset  $I' \subset \{0,\dots, N\}$ of cardinality $N-n-1$ such that $\pi_{I'}^\T (F)$ is full dimensional. Such $I'$ is a subset of some $I \in \mathcal I$ and the projection $\pi_{I}^\T (F)$ is still full dimensional. This contradicts that $F$ is inside the hypersurface $Z_I$.
\end{proof}

\begin{defn}\label{fine-conv}
A scaled sequence of $n$-dimensional varieties $X_\alpha\subset\cp^N$ {\em tropically converges}
to a weighted $n$-dimensional polyhedral complex $Y\subset\tp^N$
if the two following conditions hold:
\begin{itemize}
\item
$Y\subset\tp^N$ is a coarse tropical limit of $X_\alpha$;
\item
for any facet $F\subset Y$ its weights $w_\alpha(F)$ are equal to $w(F)$
for sufficiently large $\alpha$.
\end{itemize}
\end{defn}

\begin{prop}\label{prop:trop_conv}
A sequence of hypersurfaces $X_\alpha \subset \C\PP^N$ tropically converges to $Y\subset \T\PP^N$
in the sense of Definition \ref{fine-conv-hyp} if and only if it does
in the sense of Definition \ref{fine-conv}.
\end{prop}

\begin{proof}
Let $Y\subset\tp^N$ be the coarse tropical limit of
hypersurfaces $X_\alpha$ (required by both Definitions \ref{fine-conv-hyp} and Definition \ref{fine-conv}). 
By Proposition \ref{WY}, it is an $(N-1)$-dimensional
polyhedral complex in $\tp^N$. 
The statement follows from Proposition \ref{prop:index-weight}.
\end{proof} 

\begin{theorem}[Compactness]\label{thm-compactness}
Suppose that $X_\alpha\subset\cp^N$, $\alpha\in A$,
is a scaled sequence of $n$-dimensional varieties
of universally bounded degrees.
Then, there exists a scaled subsequence $A'\subset A$ such that $X_\alpha$, $\alpha\in A'$,
tropically converges to a {\rm(}non-empty{\rm)} balanced weighted $n$-dimensional polyhedral complex in $\tp^N$. 
\end{theorem} 

\begin{rmk}
The so-called {\em non-Archimedean amoebas},
{\it i.e.} images of non-Archimedean varieties under coordinatewise valuation maps,
may be viewed as counterparts of tropical limits in the world of non-Archimedean
algebraic geometry.
Polyhedrality properties of non-Archimedean amoebas
were discovered by Bieri and Groves
\cite{BiGr}, see also the thesis of Speyer \cite{Speyer-thesis}.
\end{rmk}

\begin{proof}
By Lemma \ref{lem-Z} we may assume (after passing to a subsequence)
that any accumulation point of $\am_\alpha$ is contained in an attractor $Z$. The weights $w^\alpha_F$ of the $n$-facets of $Z$ are {\it a priori} bounded in terms of the degrees of $X_\alpha$.  Thus, by passing to a subsequence we can assume that
these weights stabilize for large $\alpha$. Let $Y$ be the union of $n$-facets of $Z$ with positive weights, which is a (pure) $n$-dimensional polyhedral complex. 

Proposition \ref{WY} implies that the set of accumulation points of $\am_\alpha$ coincides with $Y$, that is, $Y$ is non-empty and is the coarse tropical limit of $X_\alpha$. Moreover, $Y$ is the tropical limit of $X_\alpha$ in the sense of Definition \ref{fine-conv} as the weights of the facets of $Y$ stabilize.

We now show that $Y$ is balanced.
Let $E$ be an $(n-1)$-dimensional face of $Y$ with the relative interior $E^\circ$ contained in some $\T\PP_I^\circ$. Let $F_1,\dots,F_m$ be the $n$-facets of $Y$ adjacent to $E$. Choose a small $(N-n+1)$-dimensional $I$-ball $M_E$ passing through a point in $E^\circ$ such that $\dd M$ splits as the union of membranes $M_{F_1}, \dots, M_{F_m}$ passing through $F^\circ_1,\dots,F^\circ_m$ respectively, with consistent coorientations.

Note that $\Log_{t_\alpha}^{-1}(M_E)\cap X_\alpha$ is a singular $(n+1)$-chain in $(\C^\times)^{N-I}\times \C^I$ whose  boundary is the union of $\Log_{t_\alpha}^{-1}(M_{F_i})\cap X_\alpha$, $i=1, \dots m$. Thus,  
\begin{equation}\label{eq:balanced}
\sum\limits_{j=1}^m \mathcal K^\alpha_{F_j} =0,
\end{equation}
where all classes are taken in $H_n((\C^\times)^{N-I}\times \C^I;\Z)=\Lambda^n(\Z^N/\Z^I)$. The classes of facets which are not in $\T\PP_I$ vanish in $\Lambda^n(\Z^N/\Z^I)$. Thus, \eqref{eq:balanced} is precisely the balancing condition for the facets in $\T\PP_I$.
\end{proof}

\subsection{Degree of the tropical limit}
First, we recall the definition of the degree of a balanced weighted polyhedral complex $Y \subset \T\PP^N$. 

For a point $x\in \R^N \subset \T\PP^N$,
we consider $L=L^k(x)$, the {\em fan-like linear tropical $k$-subspace} constructed as follows. 
Take the $N+1$ divisorial rays $R_0,\dots,R_N$ from $x$, 
that is, the projections of the negative coordinate rays at (any lift of) $x$ in $\T^{N+1}\setminus \{\infty\}$ to $\T\PP^N$.
The $k$-dimensional  polyhedral complex $L^k(x)\subset\tp^N$ is the closure in $\tp^N$ of the union
of the convex cones in $\R^{N}$ spanned by all possible collection
of $k$ rays from $R_0,\dots,R_N$ in $\R^N$.

The polyhedral complex $L^k(x)$ is a smooth fan with the vertex $x$. For an underlying matroid one can take $M=\{0,\dots, N\}$ with dependent subsets $I\subset M$, for all $|I|\ge k+2$.
Note that $L^k_I(x):= L^k(x) \cap \tp_I$, the restriction of $L^k(x)$ to the coordinate subspace $\tp_I \subset\tp^N$ is again a fan-like tropical linear subspace, but of dimension $k-|I|$.

We say that $L^k(x)$ intersects $Y$ transversally at $y\in  \T\PP_I$ if $y$ belongs to the relative interior of a facet $F$ of $Y$ and the relative interior of a facet $G\subset L^k_I(x)$.
In this case, $k+\dim Y=N$,  and we define the local tropical intersection number  $\iota_y(Y,L)$ as the index of the sublattice 
in $\Z^{N-I}$ generated by the lattices $F_\Z$  and $G_\Z$  of integral vectors in the faces $F\subset Y$ and $G\subset L$ times the weight $w(F)$.

\begin{lemma}
Let $Y\subset \T\PP^N$ be an $n$-dimensional weighted balanced polyhedral complex. For generic $x\in \R^N$ the linear subspace  $L=L^{N-n}(x)$ intersects $Y$ transversally at finitely many points. The total intersection number $d=\sum_{y\in L \cap Y} \iota_y(Y,L) $ is independent of $x$.
\end{lemma}
\begin{proof}
The set of $x$ such that $L=L^{N-n}(x)$ intersects $Y$ not transversally is a finite union of hypersurfaces. It happen when some facet of $Y$ meets a codimension one face of $L$ and vice versa. In each of the two cases the independence of the local intersection number under slight displacements of $x$ is a direct consequence of the balancing property of $L$ and $Y$.
\end{proof}

\begin{definition}
We say that $d$ is the {\em degree} of $Y\subset \T\PP^N$.
\end{definition}

\begin{remark}
If $Y$ is a tropical hypersurface, then its degree is the degree of a defining tropical polynomial.
\end{remark}

\begin{example}\label{example:degree1}
If $Y$ is the closure in $\T\PP^N$ of the Bergman fan $\Sigma_M \subset \R^N$,
then the degree of $Y$ is equal to $1$.
To see this,
we choose a subset  $I=\{i_1,\dots i_n\} \subset M$ of $n$ independent elements. For $L=L^{N-n}(x)$ we choose its vertex $x\in \R^N$ very close to a corner in the coordinate stratum $\T\PP_I \subset \T\PP^{N}$. Then,
the tropical intersection of $\Sigma_M$ with this $L$ is just one point $y\in \R^N$ with multiplicity $1$. 
This can be seen by successively reducing the matroid by removing elements in $I$: each time the rank of the reduced matroid decreases by 1, finally leading to a matroid of rank 1, whose Bergman fan is just one point.
\end{example} 

We now turn to comparing the degree of the complex varieties $X_\alpha$ in a scaled sequence with the degree of its tropical limit.  
First, notice that it is easy to construct a sequence of linear $k$-dimensional subspaces $\mathcal L_t \subset \cp^N$, $t\in \R_{>0}$ which tropically converges to a given $L^k(x)$. For instance, one can take the intersection of $N-k$ generic hyperplanes $H^{(j)}=  \{\sum_{i=0}^N c^{(j)}_i (t) z_i =0\}$ in $\cp^N$ such that for every $j$ the sequence of the coefficients $\{ c^{(j)}(t)\}$ tropically converges to $-x$, that is $\lim_{t\to\infty} \log_t |c^{(j)}_i (t)| = -x_i$, $j=1,\dots, N-k$.

\begin{prop}
Let $Y$ be the tropical limit of a sequence $X_\alpha\subset \C\PP^N$,
and let $L$ be the tropical limit of a sequence of linear $(N-n)$-subspaces $\mathcal L_t\subset \C\PP^N$. 
Let $y\in Y\cap L$ be a point of transversal intersection of $Y$ and $L$,
and let $U \supset y$ be a small open neighborhood.
Then, for sufficiently large $\alpha$ the intersection of $X_\alpha$
and $\mathcal L_{t_\alpha}$ in $\Log_{t_\alpha}^{-1}(U)$ is equal to $\iota_y(Y,L)$.
\end{prop}
\begin{proof}
We consider the case $y\in \R^N$, that is, $y$ is the interior point of the facets $F\subset Y$ and $G\subset L$.
The sedentary case is similar. 

For large $t_\alpha$, the intersection of $X_\alpha$ and $\mathcal L_{t_\alpha}$ in $\Log_{t_\alpha}^{-1}(U)$ 
is calculated by the Poincar\'e duality as the comparison of $[\mathcal M_F^\alpha]\wedge \Vol_G$ 
against the volume element in the lattice $\Z^N$
({\it cf}. Section \ref{section_weights}).
By Lemma  \ref{lem-weight}, this is equal to the coefficient of $w(F) \Vol_F \wedge \Vol_G$ 
at the primitive element of $\Lambda^N \Z^N$, which is, by definition, $\iota_y(Y,L)$.
\end{proof} 

As an immediate corollary we obtain the following statement.

\begin{corollary}\label{lemma:degree1}
Let $Y$ be the tropical limit of a sequence  $X_\alpha\subset \cp^N$. Then,
the degree of $Y \subset \tp^N$ equals  the degree of $X_\alpha$ for sufficiently large $\alpha$.
\qed 
\end{corollary}

\begin{rmk}
A degeneration of a smooth plane curve of
degree $d$ into $d$ generically intersecting
lines never corresponds to a smooth tropical 
limit for $d\ge 4$. If it did, 
the tropical limit would have
to be a tropical modification of a full graph on $d$ vertices 
(in this case such a tropical modification consists in adding infinite
edges at some of the existing vertices).
However, a 1-dimensional smooth $\Q$-tropical
variety in $\tp^2$ is simply a smooth planar tropical curve,
and thus has to be $3$-valent. Since, by the balancing condition,
any planar
tropical curve has to have at least one 
infinite edge,
we 
conclude that $d\le 3$. 

Meanwhile, any smooth tropical curve $Y$ of degree $d$ in $\tp^2$
corresponds to a scaled sequence of smooth curves
of degree $d$ in $\cp^2$ tropically converging to $Y$.
If 
a tropical polynomial defining $Y$ has integer coefficients, 
then this scaled sequence may be chosen to correspond
to a complex 1-parametric family of curves in $\cp^2$ 
as in Example \ref{eg:scaled_sequence}(3). We refer to $Y$ 
in the latter case as {\em defined over $\Z$}. In this
case all bounded edges of $Y$ have integer length.

If $Y$ is defined over $\Z$, then we may choose the central fiber for
the 1-parametric family to consist of $d^2$ irreducible components
corresponding to the vertices of $Y$. But we may also introduce
additional components of the central fiber by subdividing
edges of length greater than one.
More generally, a central fiber is determined by an appropriate
triangulation of the tropical limit.
We take this opportunity to remind the reader that the statement
of Theorem \ref{thm:main} does not require 
to make any choice 
of a triangulation of the tropical limit (and thus to
determine a central fiber) to define and compute tropical homology.
From its very definition, tropical homology is independent of triangulation.
\end{rmk}


\section{Degeneration of 1-parameter projective family and the limiting mixed Hodge structure}

Let $\mathcal Z \subset \C\PP^N \times \mathcal D^*$ be a complex analytic one-parameter family of projective varieties of dimension $n$
over  a punctured disc $\mathcal D^*$.
That is, $\mathcal Z$ is locally given by zeros of finitely many analytic functions $F_j(x,w)$, such that for each fixed value of $w=w_0$ the collection $\{F_j(x, w_0)\}$ defines a projective subvariety $Z_{w_0}\subset \C\PP^N$.

We can consider $\mathcal Z$ as a scaled sequence over the punctured disc with the scaling $t=|w|^{-1}$ (cf. (3) in Example \ref{eg:scaled_sequence}).
Suppose $X\subset \T\PP^N$ is the tropical limit of $\mathcal Z$ which we {\em assume} to be a smooth
projective $\Q$-tropical variety.

\begin{remark}
The assumption is not unreasonable. One can show that in the case of algebraic family $\mathcal X$ the limit always exists and is a projective $\Q$-tropical variety. For $\mathcal X$, an algebraic family of degree $d$ hypersurfaces in $\C\PP^N$,
there is a simple criterion for smoothness
of the tropical limit $X$. 
In this case, we may assume 
(after a rescaling and passing to a subsequence)
that the coefficients tropically converge
to a function $a: \Delta_d^\Z \to \Z$,
well-defined up to a constant.
The convex hull of the overgraph of $-a$
gives a polyhedral decomposition of $\Delta_d^\Z$.
Then, $X$ is smooth if and only if this decomposition
is a unimodular triangulation.
\end{remark}

We adopt the classical construction of Mumford \cite{toroidal}
of a simple normal crossing model for $\mathcal Z$.

\subsection{Unimodular triangulation of $X$}
Let $E$ be a real $N$-dimensional affine space, and let $L \subset E$ be a 
lattice of rank~$N$.
A $k$-dimensional simplex $S \subset E$ 
is called {\it $L$-primitive}
if its vertices $v_0$, $\ldots$, $v_k$ are in $L$,
and the vectors $v_1 - v_0$, $\ldots$, $v_k - v_0$
generate the intersection of $L$ with the affine span of $S$.  
A {\it cone} with
a vertex $v_0$ generated by $k$
linearly independent vectors
$u_1$, $\ldots$, $u_k$ in $E$ 
is the set
$C=\{v_0 + \sum_{i = 1}^k a_iu_i\ |\ a_i\ge 0\} \subset E$.
Such a cone $C \subset E$ is said to be $L$-primitive 
if $v_0 \in L$ and the vectors $u_1$, $\ldots$, $u_k$ can be chosen
in such a way that $v_0$, $v_0 + u_1$, $\ldots$, $v_0 + u_k$
are the vertices of a $k$-dimensional $L$-primitive simplex. 
An $n$-dimensional convex polyhedron $\Delta \subset E$ is called
$L$-primitive
if it can be represented as Minkowski sum 
$S + C$, where $S$ is an $L$-primitive simplex of dimension $0 \leq k \leq n$, 
and $C$ is an $L$-primitive cone of dimension $n - k$. 

A finite polyhedral subdivision of a convex polyhedron $\Delta' \subset E$
is called {\it convex}, if there exists a piecewise-linear convex function
$\Phi: \Delta' \to \R$ whose domains of linearity coincide with the polyhedra
of the subdivision. A finite polyhedral subcomplex of $E$ is called {\it c-extendable}
if it is a subcomplex of a convex polyhedral subdivision of $E$. 

Let $X\subset \T\PP^N$ be an $n$-dimensional 
smooth projective $\Q$-tropical variety, 
and let~$L \subset \R^N \subset \T\PP^N$ be a lattice such that
$\Z^N \supset L$. Put $X^\circ=X\cap\R^N$. 
We say that~$X$ admits a {\it unimodular triangulation}
with respect to~$L$ if there exists a finite polyhedral subdivision~$\tau$ of $X^\circ$ such that
\begin{itemize}
\item each $n$-dimensional polyhedron of $\tau$ is $L$-primitive, 
\item the asymptotic cone 
of each element of $\tau$ is generated by some (but not all) vectors
of the set $-e_1$, $\ldots$, $-e_N$, $e_1 + \ldots + e_N$, where $e_1$, $\ldots$,
$e_N$ form the standard basis of $\R^N$. 
\end{itemize} 
If such a subdivision $\tau$ can be chosen c-extendable,
we say that $X$ admits a {\it c-extendable unimodular triangulation}. 

\begin{proposition}\label{lemma:triangulation} 
Let $X\subset \T\PP^N$ be an $n$-dimensional smooth 
projective $\Q$-tropical variety. 
Then, there
exists a positive integer $m$
such that $X$ admits a c-extendable unimodular triangulation with respect to the
lattice $\frac{1}{m}\Z^N$. 
\end{proposition}
A statement somewhat similar to
Proposition \ref{lemma:triangulation} 
can also be found (albeit in a different setting)
in \cite{HeKa}. 

\begin{lemma}\label{new:triangulation}
There exists a tropical hypersurface ${\mathcal H}_X \subset \T\PP^N$ such that
\begin{itemize}
\item $X$ admits a subdivision $X'$ which is a subcomplex of ${\mathcal H}_X$
{\rm(}that is, each face of $X'$ is a face of ${\mathcal H}_X${\rm )}, 
\item the vertices of ${\mathcal H}^\circ_X = {\mathcal H}_X \cap \R^N$
belong to $\frac{1}{m}\Z^N$ for a certain positive integer $m$,
\item the asymptotic cone
of each element of ${\mathcal H}_X$ is generated by some {\rm (}but not all{\rm )} vectors 
of the set $-e_1$, $\ldots$, $-e_N$, $e_1 + \ldots + e_N$.
\end{itemize} 
\end{lemma} 

\begin{proof} 
Lemma holds tautologically if $n=\dim X=N-1$,
so assume that $n < N - 1$.
Define $\mathcal I$ to be the set of subsets $I\subset \{1,\ldots, N\}$ such that 
$\pi^\T_I(X)\subset\T\PP_{I}$ are hypersurfaces (see Section \ref{section_weights}
for the definition of the projections $\pi^\T$).   
The set ${\mathcal I}$ is non-empty; moreover, for any $i \in \{1, \ldots, N\}$,
there exists a subset $I \in {\mathcal I}$ such that $i$ does not belong to $I$. 
For each $I \in {\mathcal I}$
put $H_I = (\pi^\T_I)^{-1}(\pi^\T_I(X))$, and denote by $f_I(x_0, \ldots x_N)$ a tropical polynomial 
defining the hypersurface $H_I \subset \T\PP^N$.
For each $I \in {\mathcal I}$ put
$$
{\tilde f}_I(x_0, \ldots, x_N) =
\max\{ f_I(x_0, \ldots, x_N),
\max_{i \in I}\{\varepsilon_i + d_I x_i\}\},   
$$
where $d_I$ is the degree of $f_I$, and  $\varepsilon_i$, $i \in I$, are negative integer numbers (with sufficiently big absolute values) 
such that the hypersurface ${\tilde H}_I \subset \T\PP^N$ defined by ${\tilde f}_I$ 
satisfies the following property: 
$X \subset H_I \cap {\tilde H}_I$. We have  $X\subset \bigcap\limits_{I \in {\mathcal I}}{\tilde H}_I$.
Put
$$
{\mathcal H}_X = \cup_{I \in {\mathcal I}} {\tilde H}_I. 
$$
Since the asymptotic cone
of each element of $X$ is generated by some vectors 
of the set $-e_1$, $\ldots$, $-e_N$, $e_1 + \ldots + e_N$ (see Proposition \ref{prop:poset}), 
the same is true for the elements of ${\mathcal H}_X$. 
The coordinates of any vertex of ${\mathcal H}_X$ are solutions of systems of linear equations
with rational coefficients. Thus, the vertices of ${\mathcal H}^\circ_X = {\mathcal H}_X \cap \R^N$
belong to $\frac{1}{m}\Z^N$ for a certain positive integer $m$.
A tropical hypersurface ${\mathcal H}_X$ comes 
with a natural polyhedral subdivision
where each face is given by the subset
of the tropical monomials taking the maximal value
at its relative interior. The intersections of faces of $X$ with the faces
of ${\mathcal H}_X$ provide a subdivision $X'$ of $X$.
It remains to show that each face of $X'$ is a face of ${\mathcal H}_X$.

Assume that an $n$-dimensional face $F$ of $X'$ has a non-empty intersection
with the interior of a face $G \subset {\mathcal H}_X$ of dimension $n' > n$.
Since $F$ is contained in the support of $H_I$ for any $I \in {\mathcal I}$, 
the face $G$ is contained in a face (of dimension at least $n'$) of $H_I$ for any $I \in {\mathcal I}$. 
Pick a set $I \in {\mathcal I}$ 
such that the image $\pi^\T_I(G)$ of $G$ is full-dimensional in $\T\PP_I$.
Since $\pi^\T_I(F)$ has a non-empty intersection with the interior of $\pi^\T_I(G)$,
the face $G$ is not a face of $H_I$. 
\end{proof} 

\begin{proof}[Proof of Proposition \ref{lemma:triangulation}] 
Let ${\mathcal H}_X \subset \T\PP^N$ be a tropical hypersurface 
provided by Lemma \ref{new:triangulation}.
The closures of the connected components of the complement of ${\mathcal H}_X$ in $\R^N$
are $N$-dimensi\-onal convex polyhedral domains which form a polyhedral complex. We denote this complex by $A_X$. 
A tropical polynomial defining the hypersurface ${\mathcal H}_X$ 
provides a piecewise-linear convex function $f_X: \R^N \to \R$
whose domains of linearity coincide with the $N$-dimensional polyhedral domains of $A_X$. 
Denote by $B_X$ the polyhedral complex formed by all bounded polyhedral domains of $A_X$.
Adding to $f_X$ an appropriate piecewise-linear convex function, 
one can assume that the support $S_X$ of $B_X$ is a simplex with vertices in $\frac{1}{m}\Z^N$
and with the outward normal directions of facets given by the vectors 
$$
-e_1, \ldots, -e_N, \; e_1 + \ldots + e_N,
$$
and each unbounded $N$-dimensional polyhedron in $A_X$ is the Minkowski sum of a bounded polyhedron
of dimension $0 \leq k \leq N$ in the boundary of $S_X$ and an $(N - k)$-dimensional cone generated by some vectors 
of the set $-e_1$, $\ldots$, $-e_N$, $e_1 + \ldots + e_N$. 
According to \cite{toroidal}, there exists a positive integer $\ell$ 
such that $B_X$ admits a unimodular triangulation $\tau_X$ with respect to $\frac{1}{\ell m}\Z^N$.
Moreover, the triangulation $\tau_X$ can be chosen in such a way that
there exists a continuous function $\Phi_X: S_X \to \R$ such that
the restriction of $\Phi_X$ to any $N$-dimensional polyhedral domain $\delta$ of $B_X$
is a piecewise-linear convex function whose domains of linearity coincide
with $N$-dimensional simplices of $\tau_X$ which are contained in $\delta$.
Extend the function $\Phi_X$ to $\R^N$ in the following way: for any point $x \in \R^N \setminus S_X$,
consider the point $x' \in S_X$ which is the closest one to $x$ (with respect to the standard Euclidean distance in $\R^N$) 
among the points of $S_X$, 
and put $\Phi_X(x) = \Phi_X(x')$.
The function $\Phi_X: \R^N \to \R$ is continuous, and its restriction to any $N$-dimensional convex polyhedral 
domain $\delta$ in $A_X$ is a piecewise-linear convex function.
For a sufficiently small positive number $\varepsilon$,
the function $f_X + \varepsilon\Phi_X$ defines a convex polyhedral subdivision of $\R^N$, 
and this subdivision is a unimodular triangulation with respect to $\frac{1}{\ell m}\Z^N$.
This unimodular triangulation provides a unimodular triangulation
with respect to $\frac{1}{\ell m}\Z^N$ 
for each subcomplex of ${\mathcal H}_X$.
Thus, $X$ admits a c-extendable unimodular triangulation with respect to $\frac{1}{\ell m}\Z^N$. 
\end{proof}

\subsection{Construction of the central fiber}
Applying Proposition \ref{lemma:triangulation} (and performing a base change) we can assume that $X$ is unimodularly triangulated.

We put the mobile part $X^\circ=X\cap \R^N$ of the tropical variety $X$ in $\R^N \times \{1\} \subset \R^N \times \R_{\ge 0}$ and take the cone from the origin over it. The closure $\Sigma_X$ of this cone in $\R^N \times \R_{\ge 0}$ is a rational polyhedral (non-complete) fan of dimension $n+1$ in $\R^{N+1}$ with  unimodular faces. Thus, it defines a smooth (non-compact) toric variety $P_{\Sigma_X}$.

The fan $\Sigma_X$ maps to $\R_{\ge 0}$ along the last coordinate, thus the toric variety $P_{\Sigma_X}$ naturally maps to $\C$. The intersection of $\Sigma_X$ with the hyperplane $\R^N \times \{0\}$ coincide with the asymptotic fan $A(X)$ of $X$.  The faces $A_I$ of $A(X)$ correspond to nonempty sedentary strata $X_I$ of $X$. We denote the corresponding toric boundary strata of $P_{\Sigma_X}$ by  $D_I$. 
In particular, the rays $A_i, i=0,\dots, N$ of $A(X)$ are along the divisorial vectors of $\T\PP^N$.

The general fiber $P_\xi$ of the map $P_{\Sigma_X} \to \C$ is the toric variety associated to the asymptotic fan $A(X)$, which is a subfan of the standard fan for $\C\PP^N$. Thus, $P_\xi$ is naturally isomorphic to $\C\PP^N$ with some coordinate strata removed. 

The central fiber $P_0$ is a normal crossing divisor in $P_{\Sigma_X}$ whose components $D_\nu$ are toric varieties associated to the rays of $\Sigma_X$ through the mobile vertices $\nu$ of $X$.
More generally, the toric strata of $P_0$ are labelled by the mobile faces of $X$: if $\Delta$ is a mobile face of $X$ which is the Minkowski sum of the simplex spanned by vertices $\nu_{0}, \dots, \nu_{k}$ and the cone spanned by the divisorial vectors indexed by $I$, then the orbit $O_\Delta \cong (\C^*)^{N-m-|I|}$ is the maximal torus in the intersection $D_\Delta:= D_{\nu_{0}} \cap\dots \cap D_{\nu_{k}} \cap D_I$.

Observe that $P_{\Sigma_X}$ is quasi-projective. The polarization is  given by a piecewise-linear function provided by 
Proposition \ref{lemma:triangulation}. In particular, $P_{\Sigma_X}$ has a K\"ahler structure. 

Notice that the removed coordinate strata from $\C\PP^N$ do
not meet the fibers $Z_w\subset \mathcal Z$. Hence our family $\mathcal Z$ is embedded in $P_{\Sigma_X}$. We denote its closure  by  $\bar{\mathcal Z}$. It is now a proper family over the full disk $\mathcal D \ni 0$. Let $Z\subset \bar{\mathcal Z}$ denote the fiber over $0$.
The mobile faces of $X$ label the intersections of the family $\bar{\mathcal Z}$ with the toric strata of   $P_0$: for a mobile face $\Delta$ of $X$ we let $Z_\Delta:=\bar{\mathcal Z} \cap D_\Delta\subset Z$ and $Z^\circ_\Delta:=\bar{\mathcal Z} \cap O_\Delta \subset Z$.

\begin{lemma}\label{lemma:matroid_amoeba}
Let $\Delta$ be a mobile $k$-face of $X$. Then,
there is a compactification of $O_\Delta$ 
to the projective space $\C\PP^{N-k}$, such that the closure of $Z^\circ_\Delta$ 
is a linear subspace in $\C\PP^{N-k}$. In particular, $Z^\circ_\Delta \subset O_\Delta$ 
is isomorphic to the complement of the hyperplane arrangement corresponding 
to an underlying matroid of the relative fan $\Sigma_\Delta$.
\end{lemma} 

\begin{proof}
It is helpful to have a geometric picture. One can consider the logarithmic amoeba $\am$ of the affine part of the family $\mathcal Z \cap ((\C^*)^{N}\times \mathcal D^*)$, namely the image under the map
$$ \Log: (\C^*)^{N}\times \mathcal D^* \to \R^N \times \R_{\ge 0}, \quad (z_1,\dots, z_N, w) \mapsto (\log|z_1|, \dots, \log |z_N|,  - \log |w|)
.$$
After shrinking by $\log t$ as $t \to \infty$ this amoeba in the limit coincides with the fan $\Sigma_X$. Adding the central fiber to $\mathcal Z$ results in adding divisors at $w=0$.

We take an $(N-k)$-ball in the horizontal plane ($t=const$)  integrally transversal to the cone over $\Delta$ and intersect it with the amoeba $\am$. We define $\am_\Delta$ as the limit of this intersection as $t\to \infty$ and the size of the ball grows
linearly with $t$.

Let $M$ be a matroid whose Bergman fan $\Sigma^\circ_M \subset \R^{N-k}$ has the same support as $\Sigma_\Delta$. The vectors $e_0, \dots, e_{N-k} \in\R^{N-k}$ corresponding to the elements of $\{M\}$  define the compactification of $\R^{N-k}$ into the tropical projective space $\T\PP^{N-k}$ and the compactification of $O_\Delta$ to the projective space $\C\PP^{N-k}$. These two compactifications are compatible with the $\Log_t$ map.
 
According to the geometric picture above we can view the closure $Y_M \subset \T\PP^{N-k}$ of the fan $\Sigma^\circ_M \subset \R^{N-k}$ as the tropical limit of the constant family $Z_M$, where $Z_M$ is the closure of $Z^\circ_\Delta$ in $\C\PP^{N-k}$ (it depends on the matroid $M$ and may differs from the original closure $Z_\Delta$). The Bergman fan has the tropical degree 1, cf. Example \ref{example:degree1}.
Hence, by Corollary \ref{lemma:degree1} the degree of the subvariety $Z_M\subset \C\PP^{N-k}$ is also 1, or in other words, $Z_M$ is a linear subspace in $\C\PP^{N-k}$.
Removing the coordinate hyperplanes leads to the second statement of the lemma.
\end{proof}

We summarize the properties of our model needed later for the proof of Theorem \ref{thm:main}.

\begin{proposition}\label{thm:snc}
The family $\bar{\mathcal Z}$ is a smooth K\"ahler manifold 
{\rm (}after, perhaps, restricting
values of $w$ to a smaller disk{\rm )}.  
The central fiber $Z$ is a simple normal crossing divisor in $\bar{\mathcal Z}$ and the correspondence $\{\Delta\} \leftrightarrow  \{Z_\Delta\}$ has the following properties:
\begin{enumerate}
\item The subcomplex of $X$ formed by the finite mobile faces
is identified with the Clemens dual complex of $Z$.
\item The infinite mobile faces
of $X$ label the intersections of components of $Z$ with the toric strata $D_I$, and all these intersections 
are also simple normal crossings.
\item For any mobile face $\Delta\subset X$ the relatively
open part $Z_\Delta^\circ$ of $Z_\Delta$ is the complement of a hyperplane arrangement such that the relative fan $\Sigma_\Delta$ at $\Delta$ is the Bergman fan of the corresponding matroid.
\end{enumerate}
\end{proposition}

\begin{proof}
Note that $\bar{\mathcal Z}$ is smooth at the points of the central fiber $Z$
 and is transversal to every $D_\Delta$ since it is locally defined by linear
(degree 1) equations by Lemma \ref{lemma:matroid_amoeba}. 

Smoothness is an open condition and hence extends to some neighborhood of $Z$. In the toric variety $P_{\Sigma_X}$ the boundary divisor $\bigcup D_\nu \cup \bigcup D_i$ is simple normal crossing. Then, since $\bar{\mathcal Z}$ intersect each $D_\Delta$ transversally, the divisor $Z$ is also simple normal crossing. The K\"ahler structure on $\bar{\mathcal Z}$ is induced from $P_{\Sigma_X}$.
\end{proof}

\subsection{Limiting mixed Hodge structure}
We continue assuming that $X$ is unimodularly triangulated.  Denote by $X^{(k)}$ the collection of finite mobile $k$-faces of $X$, and let $Z^{(k)}=\bigsqcup_{\Delta\in X^{(k)}} Z_{\Delta}$ be the disjoint union of the $(k+1)$-intersections of components in $Z$.

The weight spectral sequence (cf. \cite{Steen}, Ch. 11) associated to $\mathcal Z$ has the first term
$${E}_1^{r, k-r}=\bigoplus_{l\ge \max\{0, r\}} H^{k+r-2l}(Z^{(2l-r)};\Q)[r-l], 
$$
where $[r]$ means the $r$-th Tate twist. 

The odd rows in $E_1$ are zero, we disregard them. Then, we dualize and make shifts of indices in the even rows, so that the relabeled term now reads
$$\tilde{E}^1_{q,p}:=\Hom ({E}_1^{q-p, 2p}; \Q)
= \bigoplus_{\tilde l=0}^{\min\{p, q\}} H_{2\tilde l}(Z^{(p-2 \tilde l+q)}; \Q)[p- \tilde l].
$$
Here is the beginning of the new $\tilde{E}^1$ term (where $H_{2l}(k)[r]$ means $H_{2l}(Z^{(k)}; \Q)[r]$):
$$ \xymatrix{
 H_0(2)[2] &
 \txt{$H_0(3)[2] $ \\ $\oplus H_2(1)[1]$}  \ar[l]_{d} &
 \txt{$H_0(4)[2] $ \\ $\oplus H_2(2)[1]$ \\ $\oplus H_4(0) $}  \ar[l]_{d} &  &
\\
H_0(1)[1] &
 \txt{$H_0(2) [1]$ \\ $\oplus H_2(0)$}  \ar[l]_{d}  &
 \txt{$H_0(3) [1]$ \\ $\oplus H_2(1)$}  \ar[l]_{d}  &
 \txt{$H_0(4) [1]$ \\ $\oplus H_2(2)$}  \ar[l]_{d} &
\\
H_0(0) & H_0(1) \ar[l]_{d}   & H_0(2) \ar[l]_{d}  & H_0(3)  \ar[l]_{d} & H_0(4) \ar[l]_{d}
 }
$$
The differential $d=i_*+\gys$ consists of the pushforward map $i_*$ and the Gysin map
(see, e.g., \cite{Zo06}, p. 231): 
\begin{equation}\label{eq:d1}
\begin{split}
 i_* :H_{2l}(k)[r] & \to H_{2l}(k-1)[r] \\
 \gys: H_{2l}(k)[r] & \to H_{2l-2}(k+1)[r+1].
\end{split}
\end{equation}

\begin{theorem}\label{theorem:main}
For each $p$ the row complex $(\tilde{E}^1_{\bullet,p}, d)$ in the weight spectral sequence
is quasi-isomorphic to the tropical cellular chain complex $C_\bullet(X;\F_p)$. 
In particular,  $\tilde{E}^2_{q,p}\cong H_{q}(X;\F_p)$.
\end{theorem} 

The proof of Theorem \ref{theorem:main} is presented in the next section.

\begin{proof}[Proof of Theorem \ref{thm:main} and Corollary \ref{cor:main}]
The weight spectral sequence degenerates at $E_2$ abutting to cohomology of the canonical fiber $\mathcal Z_\infty$ with the monodromy weight filtration (cf. \cite{Steen}, Ch. 11).
Thus, Theorem \ref{thm:main} follows from Theorem \ref{theorem:main}. 

Since all closed strata in $Z$ are composed of complements of hyperplanes all even cohomology groups $H^{2i}(Z_\Delta;\Q)$ are of $(i,i)$-type (cf. e.g. \cite{DKh}). Cohomology in odd degrees vanish. In particular, the MHS is of Hodge-Tate type: in the weight filtration only even associated graded pieces are non-trivial and each contains only Hodge $(p,p)$-type. Thus, the weight filtration calculates the Hodge numbers of the canonical fiber (and hence also of a smooth fiber). That is, $h^{p,q} (Z_w) = \dim E_2^{q-p, 2p}= \dim \tilde{E}^2_{q,p} = \dim H_{q}(X;\F_p)$.
\end{proof}


\section{Proof of Theorem \ref{theorem:main}}
The proof goes as follows. For each $p$ we introduce a double complex $(K^{(p)}_{\bullet, \bullet}, \partial, \delta)$ and calculate homology of the total complex $(K^{(p)}_{ \bullet}, \partial+\delta)$ in two ways. First, we take the $\delta$-homology, and recover the tropical cellular chain complex $C_\bullet(X;\F_p)$ (see Proposition \ref{prop:gysin}). Second, 
we introduce a filtration $F_m$ on $(K^{(p)}_{ \bullet}, \partial+\delta)$  
such that the first term of the resulting spectral sequence is a single row which coincides with 
the weight spectral sequence complex $(\tilde{E}^1_{\bullet,p}, i_*+\gys)$ (see Proposition \ref{prop:SI}). 
In the notations below
\begin{equation}\label{eq:E^1}
\tilde{E}^1_{m,p}=\oplus_{l=0}^{\min \{m,p\}} \bigoplus H_{2l}(\Delta),
\end{equation}
where $\Delta$ runs over finite mobile $(p-2l+m)$-dimensional faces of $X$, and we disregard the Tate twist. 
Theorem  \ref{theorem:main} then follows.

\subsection{Notations}
From now on we assume that all faces of $X$ are oriented. Recall that to any mobile face $\Delta$ of $X$ we can associate $Z_\Delta$, the closed subset of $Z$ which is the intersection of the corresponding components of $Z$ and some toric divisors in $P_{\Sigma_X}$. Recall also our notation $\Delta \prec^s_j \Delta'$ (and $\Delta' \succ^s_j \Delta$) when $\Delta$ is a face of $\Delta'$ of codimension $j$ and cosedentarity $s$.
We omit the superscript $s$ in case $s=0$. Here are some more notations.

\begin{itemize}
\item If $\Delta$ is a mobile face of $X$, we set $H_{2l}(\Delta) : = H_{2l}(Z_{\Delta};\Q)$.
\item For $\Delta'\succ_1\Delta$, a consistently oriented pair of mobile faces, $i_*: H_{2l}(\Delta') \to H_{2l}(\Delta)$ is the pushforward map, and $\gys:  H_{2l}(\Delta) \to H_{2l-2}(\Delta')$ is the Gysin map.
\item If $\Delta$ is sedentary, we set $H_{2l}(\Delta):=H_{2l}(\Delta_0)$, where $\Delta_0$ is the parent of $\Delta$.
\item For any face $\Delta$ of $X$ we set $W_r(\Delta) : = \wedge^r \Q\<\Delta\>$ to be the space of rational $r$-polyvectors in the linear span of $\Delta$.
\end{itemize}

We extend the meaning of the pushforward and the Gysin maps for sedentary faces.
If $\Delta'\succ_1\Delta$ then the pushforward  $i_*: H_{2l}(\Delta') \to H_{2l}(\Delta)$ is the same as $i_*: H_{2l}(\Delta'_0)\to H_{2l}(\Delta_0)$ for their parents (under the identifications $H_{2l}(\Delta)=H_{2l}(\Delta_0)$ and $H_{2l}(\Delta')=H_{2l}(\Delta_0')$). And similar for the Gysin map $\gys:  H_{2l}(\Delta) \to H_{2l-2}(\Delta')$.

If $\Delta'\succ^1_1\Delta$, then $i_*: H_{2l}(\Delta') \to H_{2l}(\Delta)$ 
is the identity (both groups equal $H_{2l}(\Delta_0)$). 
The Gysin map $\gys:  H_{2l}(\Delta) \to H_{2l-2}(\Delta')$ in this case is zero.

For a pair $\Delta'\succ_{j+s}^s\Delta$ we define the {\em residue} map
$$\shad_{\Delta'\succ\Delta}: W_r(\Delta')\to W_{r-j}(\Delta)
$$
as follows. Let $\Delta''$ be the smallest face between $\Delta$ and $\Delta'$ of the same sedentarity as $\Delta'$. 
Then, there is a canonical primitive covolume $j$-form $\Omega_{\Delta'\succ_j\Delta''}$ in $\Delta'$ 
which vanishes on polyvectors divisible by vectors in $\Delta''$. Now for a polyvector $w \in W_r (\Delta')$ we define its residue 
$\shad_{\Delta'\succ\Delta}(w) \in W_{r-j}(\Delta)$  
to be the projection to $W_{r-j}(\Delta)$ of the evaluation of $w$ on $\Omega_{\Delta'\succ_j\Delta''}$ (It is zero if $r<j$). 

The most important cases are $\Delta'\succ_1\Delta$ and $\Delta'\succ^1_1\Delta$. In the first case $\shad(w)$ is the evaluation of $w$ on the canonical linear from which defines the facet $\Delta\prec \Delta'$. In the second case $\shad(w)$ is just the projection.

Let $\Delta$ be a $k$-dimensional face of $X$, maybe infinite and sedentary. It will be convenient to treat vertices and divisorial vectors which span $\Delta$ on an equal footing. Namely, for $\Delta$ we write a sequence $(\nu_0 \nu_1\dots \nu_k)$, where  $\nu_0$ is a vertex and each $\nu_{j \ne 0}$ may denote a vertex or a divisorial vector of $\Delta$. 

The orientation of $\Delta=(\nu_0 \nu_1\dots \nu_k)$ is encoded in the sign order of the sequence. 
To get consistent signs in the Gysin and pushforward  maps for the differential in the weight spectral sequence we use the following convention. For the Gysin map we add a new vertex (or a divisorial vector) at the end of the old sequence. For the pushforward  map we remove the last element (after reordering the sequence if needed).

\begin{lemma}\label{lemma:commute}
For a class $\beta\in H_{2l}(\Delta)$ we denote by $\beta^q$ its image in $H_{2l-2}(\Delta \nu_q)$ under the Gysin map, and by  $\beta_j$ its image in $H_{2l}(\Delta \setminus \nu_j)$ under the pushforward  map. Then
\begin{gather*}
(\beta_j)_i=(\beta_i)_j, \quad  (\beta^r)^q=(\beta^q)^r, \\ 
(\beta_j)^q=(\beta^q)_j, \quad \sum_q \beta^q_q + \sum_j \beta_j^j = 0,
\end{gather*}
where the sums in the last identity are over all $\nu_j \in \Delta$ and all ${\nu_q\in\Link(\Delta)}$. 
{\rm (}Here and later ${\nu_q\in\Link(\Delta)}$ means $(\Delta \nu_q) \succ_1 \Delta${\rm )}. 
\end{lemma}

\begin{proof}
The first two identities follow from writing $i_*^2=0$ and $\gys ^2 = 0$ in components. 

The other two follow from the (anti-)commutative diagram $i_* \gys +\gys  i_*=0$:
\begin{equation}
\xymatrix{
H_{2l} (\nu_0\dots \nu_k) \ar[d]^{\gys} \ar[r]^{i_*} & \sum_j H_{2l} (\nu_0\dots \hat \nu_j \dots \nu_k) \ar[d]^{\gys}\\
\sum_q H_{2l-2} (\nu_0\dots \nu_k\nu_q) \ar[r]^--{i_*} & H_{2l-2} (\nu_0\dots \nu_k) \oplus \sum_{j, q} H_{2l-2} (\nu_0\dots \hat \nu_j \dots \nu_k\nu_q)
}
\label{eq:gysin}
\end{equation}
In $H_{2l-2} (\nu_0\dots \hat \nu_j \dots \nu_k\nu_q)$ we have $(\beta_j)^q-(\beta^q)_j=0$ (with our sign convention).
The last identity 
$$
\sum_q \beta^q_q + \sum_j \beta_j^j = 0
$$
takes place in $H_{2l-2} (\nu_0\dots \nu_k)$. The meanings of $\beta_j^j$ and $\beta^q_q$ are unambiguous: they only makes sense in one order: $\beta_j^j = (\beta_j)^j$ and $\beta_q^q= (\beta^q)_q$. 
\end{proof}

\subsection{The double complex $(K^{(p)}_{\bullet, \bullet}, \partial, \delta)$}
For any face $\Delta\in X$ we set
$$A^{(p)}_l(\Delta)=\bigoplus_{\Delta'\in \Star (\Delta)} H_{2l}(\Delta')\otimes W_{p-l}(\Delta')/\sim
$$
where the equivalence is defined as follows.
For a pair of simplices $\Delta''\prec_j \Delta'$ we have two maps $i_*$ and $\iota$ (the inclusion of polyvectors) going into opposite directions:
$$\xymatrix{
H_{2l}(\Delta'')  \otimes W_{p-l}(\Delta'') \ar@<2pc>[d]^{\iota} \\
H_{2l}(\Delta') \ar@<2pc>[u]^{i_*} \otimes W_{p-l}(\Delta')
}
$$
We identify elements $\beta\otimes\ \iota (w) \in H_{2l}(\Delta')\otimes W_{p-l}(\Delta')$ and $i_*(\beta) \otimes w \in H_{2l}(\Delta'')\otimes W_{p-l}(\Delta'')$.

\begin{remark}
To generate the space $A^{(p)}_j(\Delta)$ it is sufficient to consider $\Delta'\succ_j\Delta$ with $j$ between 0 and $p-l$. 
\end{remark}

Now we set
$ K^{(p)}_{k,l}:=\bigoplus_{\dim\Delta=k} A^{(p)}_l(\Delta),
$
and define two differentials
$$\partial: K^{(p)}_{k,l} \to K^{(p)}_{k-1,l},\quad \delta: K^{(p)}_{k,l} \to K^{(p)}_{k,l-1}
$$
as follows. 

The horizontal differential $\partial$ acts essentially just as the boundary map on the cell complex $X$.
Namely, if  $\Delta''\prec_1\Delta$ is a consistently oriented pair, 
then $\partial: A^{(p)}_l(\Delta) \to A^{(p)}_l(\Delta'')$ 
acts on representatives $\beta\otimes w \in H_{2l}(\Delta')\otimes W_{p-l}(\Delta')$ as the identity. 
Clearly, it respects the equivalence. 

If  $\Delta''\prec^1_1\Delta$, then by Proposition  \ref{prop:poset} for any face $\Delta' \in \Star (\Delta)$
there is a corresponding face $\Delta'''\in \Star (\Delta'')$  in the same family as $\Delta'$. That is, $H_{2l}(\Delta''')=H_{2l}(\Delta')$, and we let $\pi:  W_{p-l}(\Delta') \to W_{p-l}(\Delta''')$ be the natural projection. Then, 
we define the image of $\beta\otimes w$ in $A^{(p)}_l(\Delta'')$ to be $\beta\otimes \pi(w)$. 
Since inclusions $\iota$ for the spaces $W_{p-l}(\bullet)$ commute with the projections $\pi$, 
any equivalence between representatives $\beta\otimes w$ in $A^{(p)}_l(\Delta)$ 
also holds for their images in $A^{(p)}_l(\Delta'')$.

We define the vertical differential $\delta: A^{(p)}_l(\Delta) \to A^{(p)}_{l-1}(\Delta)$ by a certain combination of the pushforward $i_*$ and  the Gysin maps. Namely, given $\beta\otimes w \in H_{2l}(\Delta')\otimes W_{p-l}(\Delta')$ we first choose a representing sequence $(\nu_0 \nu_1\dots \nu_{k'})$ for $\Delta'$ (in particular, a reference vertex $\nu_0\in \Delta'$) and then set
\begin{equation}\label{eq:vertical}
\delta(\beta\otimes w) := \sum_{\nu_q\in\Link(\Delta')} \beta^q \otimes ({\nu_{0q}}\wedge w) + \sum_{\nu_j\in \Delta'} \beta_j^j \otimes (\nu_{0 j}\wedge w).
\end{equation}
Here, $\nu_{0q}$ means the vector $(\nu_0 \nu_q)$ if $\nu_q$ is a vertex, and $\nu_{0q} = \nu_q$ if $\nu_q$ is a divisorial vector, and similar for $\nu_j$. Note that the first sum is in $\sum_q H_{2l-2}(\Delta' \nu_q)\otimes W_{p-l+1}(\Delta' \nu_q)$ 
and the second sum is in $H_{2l-2}(\Delta')\otimes W_{p-l+1}(\Delta')$, 
but both represent elements in $A^{(p)}_{l-1}(\Delta)$. 

\begin{lemma}
The map $\delta: A^{(p)}_j(\Delta) \to A^{(p)}_{j-1}(\Delta)$ is well defined.
\end{lemma}

\begin{proof}
There are two things to check: (1) it is independent of the choice of the representing sequence $\Delta'=(\nu_0 \nu_1\dots \nu_{k'})$, of which only independence of the choice of the reference vertex $\nu_0\in\Delta'$ is non-trivial, and (2) it is defined on the equivalence class of  $\beta\otimes w$. Since the definition \eqref{eq:vertical} is linear in $w$ we can let $w=1$.

For (1) let $\nu_0'=\nu_1\in \Delta'$ be a new reference vertex ($\nu_1$ is not a divisorial vector). Then
\begin{multline*}
\sum_{\nu_q\in\Link(\Delta')} \beta^q \otimes \nu_{0q} + \sum_{\nu_j\in \Delta'} \beta_j^j \otimes \nu_{0j}
-  \sum_{\nu_q\in\Link(\Delta')} \beta^q \otimes \nu_{1q} - \sum_{\nu_j\in \Delta'} \beta_j^j \otimes \nu_{1j}\\
= \sum_{\nu_q\in\Link(\Delta')} \beta^q \otimes \nu_{01} + \sum_{\nu_j\in \Delta'} \beta_j^j \otimes \nu_{01}\\
\sim \sum_{\nu_q\in\Link(\Delta')} \beta^q_q \otimes \nu_{01} + \sum_{\nu_j\in \Delta'} \beta_j^j \otimes \nu_{01},
\end{multline*}
which is 0 by the last identity in Lemma \ref{lemma:commute}.

For (2) pick an equivalent representative $\beta_r\in H_{2l}(\Delta' \setminus \nu_r)$ (assuming, of course, $\Delta' \setminus \nu_r\succ \Delta$) and compare
$$
\delta \beta= \sum_{\nu_q\in\Link(\Delta')} \beta^q \otimes \nu_{0q}  + \sum_{\nu_j \in \Delta', \ j\ne r} \beta_j^j \otimes \nu_{0j} +\beta_r^r \otimes \nu_{0r},
$$
with
$$
\delta (\beta_r) = \sum_{\nu_q\in\Link(\Delta' \setminus \nu_r)} \beta_r^q \otimes \nu_{0q} +\beta_r^r \otimes \nu_{0r} + \sum_{\nu_j \in \Delta' \setminus \nu_r} (\beta_r)_{j}^j \otimes \nu_{0j}.
$$
First, observe that in the first sum of the latter expression the Gysin image of those $\nu_{q}\in \Link (\Delta'\setminus \nu_r)$ which are not in $\Link (\Delta')$ must be zero by the third identity in Lemma \ref{lemma:commute}. Furthermore, we have
$$(\beta_r)_j^j = (\beta_{rj})^{j} = - (\beta_{jr})^{j} = -((\beta_j)_r)^j = ((\beta_j)^j)_r =  (\beta_j^j)_r.
$$
Then, the two expressions are manifestly equivalent. 
\end{proof} 

\begin{lemma}
The triple $(K^{(p)}_{\bullet, \bullet}, d, \delta)$ is a double complex. 
\end{lemma}
\begin{proof}
The horizontal differential $\partial$ is the standard boundary map on the simplicial complex. The fact that  $\partial$ commutes with $\delta$ is clear from the definitions. Finally, $\delta^2 =0$ follows immediately from the second and the third commutativity identities in Lemma \ref{lemma:commute}.
\end{proof} 

\subsection{Filtration on $(K^{(p)}_{k, \bullet}, \delta)$ by columns}

\begin{proposition}\label{prop:gysin}
The pair $(K^{(p)}_{k, \bullet}, \delta)$ is a resolution of $C_k(X;\F_p)$ 
and the resulting homology complex $(H_0(K^{(p)}_{\bullet, \bullet}, \delta), \partial)$ 
is isomorphic to $C_\bullet(X;\F_p)$.
\end{proposition} 

To prove the proposition we introduce an increasing filtration $F^\Delta$ 
on the spaces $A^{(p)}_l(\Delta)$:
\begin{equation}\label{eq:filtration}
0\subseteq F^\Delta_{-p+l} \subseteq \dots \subseteq F^\Delta_0 = A^{(p)}_l(\Delta),
\end{equation}
where $F^\Delta_{-r} A^{(p)}_l(\Delta)$ consists of elements which can be represented by $\beta\otimes w \in H_{2l}(\Delta')\otimes W_{p-l}(\Delta')$ with $\Delta'\succ_{\le (p-l-r)} \Delta$. 
Or, equivalently, $\beta\otimes w$ is in $F^\Delta_{-r} A^{(p)}_l(\Delta)$ if  $w$ belongs to the ideal generated by $\iota (W_r(\Delta))$. 
In particular, the filtration on $A^{(p)}_0(\Delta)$ induces one on the coefficient groups $\F_p(\Delta)$.

\begin{lemma}\label{lemma:filtration}
The associated graded groups
$\Gr_{-r}^{F^\Delta} A^{(p)}_l(\Delta) = F^\Delta_{-r} A^{(p)}_l(\Delta)/  F^\Delta_{-r-1} A^{(p)}_l(\Delta)$ can be naturally identified with
$\bigoplus_{\Delta'\succ_{p-r-l}\Delta} H_{2l}(\Delta') \otimes W_{r}(\Delta)$.
The graded pieces $\Gr_{-r}^{F^\Delta} \F_p(\Delta)$ can be naturally identified with $\overline\F_{p-r}(\Delta) \otimes W_{r}(\Delta)$, where $\overline\F_{p-r}(\Delta)$ are the relative coefficient groups at $\Delta$ {\rm(}cf. Section \ref{sec:cellular}{\rm)}.
\end{lemma}

\begin{proof}
Consider an element in $A^{(p)}_l(\Delta)$ which is represented by  $\beta\otimes w$ with $w \in W_{p-l}(\Delta')$ for some $\Delta'\succ_j\Delta$. This means that this element is in $F^\Delta_{-p+j+l} A^{(p)}_l(\Delta)$. The residue map 
$$\shad_{\Delta'\succ\Delta}: W_{p-l}(\Delta')\to W_{p-l-j}(\Delta)
$$ 
is surjective, and $\shad_{\Delta'\succ\Delta}(w)$ vanishes if and only if $w \in W_{p-l}(\Delta')$ lies in the ideal generated by $\iota (W_{p-l-j+1}(\Delta))$, that is, if and only if
$\beta\otimes w \in F^\Delta_{-p+j+l-1} A^{(p)}_l(\Delta)$. Thus, the residue map
$$\beta\otimes w \mapsto \beta\otimes \shad_{\Delta'\succ\Delta}(w)
$$
provides the desired isomorphism
$$F^\Delta_{-r} A^{(p)}_l(\Delta)/  F^\Delta_{-r-1} A^{(p)}_l(\Delta) \cong
\bigoplus_{\Delta'\succ_{p-r-l}\Delta} H_{2l}(\Delta') \otimes W_{r}(\Delta).
$$

Finally, the induced filtration on $\F_p(\Delta)$ is given by considering the polyvectors from $W_p(\Delta')$ as elements in the space $\wedge^p \Q^N$ of ambient polyvectors. The last statement of the Lemma follows.
\end{proof}

\begin{proof}[Proof of Proposition \ref{prop:gysin}]
The horizontal differentials $\partial$ on $A^{(p)}_0(\Delta)$ and the cell boundary maps on $\F_p(\Delta)$ commute with the augmentation maps  $A^{(p)}_0(\Delta)\to\F_p(\Delta)$. Thus,
it suffices to check that for each $\Delta$ the augmented complex $A^{(p)}_\bullet(\Delta)\to\F_p(\Delta) \to 0$ is exact. 
It is enough to show exactness on the associated graded level with respect to the filtration $F^\Delta$. 
The filtration $F^\Delta$ respects the vertical differential $\delta$, since $\delta$ 
raises the degree of a polyvector $w$ by 1, but puts it in a face of at most one dimension higher. 
It also respects the augmentation map $A^{(p)}_0(\Delta)\to\F_p(\Delta)$, 
and hence defines a filtration on the augmented complex. 

By Lemma \ref{lemma:filtration} for each $r$ the associated graded complex $\Gr_{-r}^{F^\Delta} A^{(p)}_\bullet(\Delta)\to\Gr_{-r}^{F^\Delta} \F_p(\Delta) \to 0$ has the form
\begin{multline*}
H_{2p-2r}(\Delta) \otimes W_{r}(\Delta)
\to \bigoplus_{\Delta'\succ_1\Delta}H_{2p-2r-2}(\Delta') \otimes W_{r}(\Delta)
\to \dots \\
 \to \bigoplus_{\Delta'\succ_{p-r}\Delta}H_{0}(\Delta') \otimes W_{r}(\Delta)
 \to
\overline\F_{p-r}(\Delta) \otimes W_{r}(\Delta)  \to 0,
\end{multline*}
where the differential is $W_{r}(\Delta)$-linear. From the definition of the vertical map $\delta$ in the double complex we can identify this differential on the homology groups as the Gysin map. 

Notice that if $\Delta$ is sedentary with a parent $\Delta_0$, then there is an obvious map 
$$\Gr_{-r}^{F^\Delta_0} A^{(p)}_\bullet(\Delta)\to\Gr_{-r}^{F^\Delta} A^{(p)}_\bullet(\Delta),
$$
which is the identity on the corresponding homology factors $H_{2l}(\Delta'_0)= H_{2l}(\Delta')$ tensored with the projection $ W_{r}(\Delta_0) \to  W_{r}(\Delta)$. Taking into account $\overline\F_{p-r}(\Delta_0)= \overline\F_{p-r}(\Delta)$ this projection extends to the augmentation
$$\overline\F_{p-r}(\Delta_0) \otimes W_{r}(\Delta_0) \to \overline\F_{p-r}(\Delta) \otimes W_{r}(\Delta) .
$$
Thus, we are reduced to proving for each $p$ the exactness of the complex 
$$
\bigoplus_{\Delta'\succ_{p-\bullet}\Delta}H_{2\bullet }(\Delta')
 \to
\overline\F_{p}(\Delta) \to 0
$$
for mobile faces $\Delta$ of $X$.

Recall Deligne's spectral sequence \cite{Deligne} which calculates the weight filtration of the mixed Hodge structures for smooth quasi-projective varieties $Y$. In our case, $Y=Z^\circ_\Delta$ is the relatively open part of $Y^{(0)}=Z_\Delta$. 
Let $Y^{(j)}$ denote the disjoint union of the intersections $Z_{\Delta'}$ for $\Delta'\succ_j\Delta$. 
Then, 
we have $H_{2l}(Y^{(j)};\Q)=\bigoplus_{\Delta'\succ_j\Delta}H_{2l}(\Delta')$. 

On the other hand, by construction $Y$ is  (cf. Proposition \ref{thm:snc}) the complement of a hyperplane arrangement in $\PP^{n-k}$ and its cohomology $H^{p}(Y;\Q)$ is isomorphic to $\overline\F^{p}(\Delta)$ by Theorem \ref{OS}.

The $E_1$ term of the Deligne's MHS spectral sequence looks as follows (in our case all odd rows are zero):
$$
\begin{array}{ccc}
\dots & \dots & \dots \\
H^0(Y^{(2)})  \to & H^2(Y^{(1)}) \to &  H^4 (Y^{(0)})\\
& H^1(Y^{(1)}) \to  & H^3 (Y^{(0)})\\
& H^0(Y^{(1)}) \to  & H^2 (Y^{(0)})\\
& & H^1 (Y^{(0)})\\
&  & H^0 (Y^{(0)})\\
\end{array}
$$
The differential is the Gysin map, and the spectral sequence degenerates at $E_2$.  The MHS weight filtration on the diagonals in $E_2=E_\infty$ is given by the rows. The MHS structure on $H^{p}(Y;\C)$ is pure of type $(p,p)$ (cf. \cite{Shapiro}). Thus, in $E_2$ the only non-zero terms will be the groups $\overline\F^{p}(\Delta)$ in the left most entries of the even rows. 
Dualizing all groups and inverting arrows in the $2p$-th row we get the desired exact complex
\begin{equation*}\label{eq:left_column}
0 \leftarrow \overline\F_p(\Delta) \leftarrow H_0(Y^{(p)}) \leftarrow \dots  \leftarrow   H_{2p-2}(Y^{(1)})  \leftarrow H_{2p}(Y^{(0)}).
\end{equation*}
\end{proof}

\subsection{Another filtration on the total complex $K^{(p)}_\bullet$ of $K^{(p)}_{\bullet, \bullet}$}

The new increasing filtration $F_m$ on $K^{(p)}_\bullet$ is, in fact, the old filtration $F^\Delta$ on the spaces $A^{(p)}_l(\Delta)$ with some degree shifts depending on $\Delta$. Namely, if $\Delta$ has dimension $k$ and sedentarity $s$ we set $F_m {A^{(p)}_l(\Delta)} = F^\Delta_{m-k-l-s} A^{(p)}_l(\Delta)$. We will see below that both $\delta$ and $\dd$ respect the filtration.

\begin{example}
We illustrate first few terms in $\Gr_m^F K^{(p)}_{\bullet, \bullet}$ for $p=2$:
$$ \xymatrix{
 H_4(0) W_0(0)_2 \ar[d]  &  H_4(1)  W_0(1)_3 \ar[l] \ar[d]  &  H_4(2)  W_0(2)_4 \ar[l] \ar[d]  &  H_4(3)  W_0(3)_5 \ar[l] \ar[d]
 &  \ar[l]
 \\
 H_2(1)W_0(0)	_1						 \ar[d] &
 \txt{$H_2(2) W_0(1)_2 $ \\ $\oplus H_2(1)W_1(1)_1$} \ar[l] \ar[d] &
 \txt{$H_2(3) W_0(2)_3 $ \\ $ \oplus H_2(2) W_1(2)_2$} \ar[l] \ar[d]  &
 \txt{$H_2(4) W_0(3)_4 $ \\ $ \oplus H_2(3) W_1(3)_3$} \ar[l] \ar[d]
  &  \ar[l]
  \\
 H_0(2) W_0(0)_0 							 &
 \txt{$H_0(3) W_0(1)_1 $ \\ $\oplus H_0(2) W_1(1)_0$}  \ar[l]  &
 \txt{$H_0(4) W_0(2)_2 $ \\ $\oplus H_0(3) W_1(2)_1 $ \\ $\oplus H_0(2) W_2(2)_0$}\ar[l]  &
 \txt{$H_0(5) W_0(3)_3 $ \\ $\oplus H_0(4) W_1(3)_2 $ \\ $\oplus H_0(3) W_2(3)_1$}\ar[l]
 &  \ar[l]
 }
$$
The notation $H_{2l}(k') W_r (k)$ means the direct sum of terms $H_{2l}(\Delta') \otimes W_r(\Delta)$ running over all incident pairs $\Delta'\succ\Delta$ of dimensions $k'$ and $k$ respectively and equal sedentarity. The subscript indicates the grading index $m$, assuming 0 sedentarity.
\end{example}

The $E^0$-term in the associated spectral sequence is
$$ E^0_{m,r}= \Gr_{m}^{F} K^{(p)}_{r+m} = \oplus_{l=0}^\infty \bigoplus_{\Delta'\succ\Delta} H_{2l}(\Delta') \otimes W_{r+s}(\Delta),
$$
where $\Delta'\succ\Delta$ run over all incident pairs of dimensions $(p-2l+m-s)$ and $(r-l+m)$, respectively, and $s$ is their sedentarity. We can replace $\Delta'$ in $H_{2l}(\Delta')$ by its parent $\Delta'_0$, whose dimension is $(p-2l+m)$. Note that the terms in the $l$-sum are zero unless $l\le p$ (original bound in $A^{(p)}_l(\Delta)$) and $l\le m$ (follows from $r+s \le r-l+m$).

We calculate the differential $d_0=\partial_0+\delta_0$. 
Pick an element in $\Gr^F_m A^{(p)}_l(\Delta)$ represented by $\beta\otimes w \in H_{2l}(\Delta') \otimes W_{p-l}(\Delta')$.
Note that $\delta$ lowers $l$ by 1 and it does not change anything else. 
Hence, it always lowers the total degree by 1 and will appear only in the $E^1$ term. 
That is, $\delta_0$ vanishes and $\delta_1$ 
is the Gysin map (cf. the proof of Proposition \ref{prop:gysin}). 

Let us analyze the effect of the horizontal map $\partial: A^{(p)}_l(\Delta) \to A^{(p)}_l(\Delta'')$. Note that $\partial$ lowers the dimension of $\Delta$ by 1. On the other hand, it can raise the sedentarity by 1 or it can lower $r$ by 1 in $F^\Delta_{-r}$ filtration. In the latter case $\dd_0$ will be zero, and $\dd$ will only contribute to $\dd_1$ in the $E^1$ term. We give some details.

If $\Delta''\prec^1_1\Delta$, then consider the corresponding face $\Delta'''\prec^1_1\Delta'$ with $H_{2l}(\Delta''')=H_{2l}(\Delta')=H_{2l}(\Delta'_0)$ and the commutative diagram of the residue maps:
$$ \xymatrix{
\beta\otimes \pi(w) \ar[d]
 &  \beta\otimes w  \ar[l]_{\partial= \id \otimes \pi} \ar[d]
 \\
\beta\otimes \shad_{\Delta'''\succ \Delta''} (\pi(w))
 &  \beta\otimes \shad_{\Delta'\succ \Delta} (w) \ar[l]_--{\partial_0}
 }
$$
Since $\shad_{\Delta'''\succ \Delta''} (\pi(w)) = \shad_{\Delta'\succ \Delta''}(w) = \pi(\shad_{\Delta'\succ \Delta} (w))$, we see that  on the associated graded level $\partial$ acts $H_{2l}(\Delta'_0)$-linearly and as the projection $\pi$ on $W_{r+s}(\Delta)$:
$$\partial_0 = \id \otimes \pi: H_{2l}(\Delta'_0) \otimes W_{r+s}(\Delta) \to H_{2l}(\Delta'_0) \otimes W_{r+s}(\Delta'').
$$

If $\Delta''\prec_1\Delta$, the commutative diagram of the residue maps
$$ \xymatrix{
\beta\otimes w \ar[d]_{\id\otimes\shad}
 &  \beta\otimes w \ar@{=}[l]_{\partial =\id} \ar[d]^{\id\otimes\shad}
 \\
 \beta\otimes \shad_{\Delta'\succ \Delta''} (w)
 &  \beta\otimes \shad_{\Delta'\succ \Delta} (w) \ar[l]_--{\dd_0}
 }
$$
shows that $\partial_0$ acts as before by the identity on $H_{2l}(\Delta')$, and by the residue map on $W_{r+s}(\Delta)$:
$$\partial_0 = \id \otimes  \shad_{\Delta\succ \Delta''}: H_{2l}(\Delta'_0) \otimes W_{r+s}(\Delta) \to H_{2l}(\Delta'_0) \otimes W_{r+s-1}(\Delta'').
$$

\begin{proposition}\label{prop:SI}
The $E^1$ term of spectral sequence associated to the filtration $F$ on $K^{(p)}_\bullet$ is a single 0-th row which coincides with the complex \eqref{eq:E^1}. 
\end{proposition}

To calculate the homology of the complex $( \Gr_{m}^{F} K^{(p)}_\bullet, \partial_0)$ we need a bit of Koszul-type linear algebra. Let $\Delta'_0$ be a mobile face of $X$. Consider the {\em residue} complex 
$$ \mathbb W_\bullet (\Delta'_0) = \bigoplus W_{\bullet+s} (\Delta),
$$
where $\Delta$ runs over all faces of $\Delta'_0$ (including sedentary ones, here $s$ is the sedentarity of $\Delta$). The differential is the residue map 
$$\shad_{\Delta\succ\Delta''}: W_r(\Delta) \to W_{r-1+s}(\Delta'')
$$
for pairs $\Delta\succ_1^s\Delta''$. The residue complex splits into the direct sum of complexes $\mathbb W_\bullet^q (\Delta'_0)$ according to the degree $q=\dim (\Delta) +\operatorname{sed}(\Delta)-r$ of its terms $W_{r} (\Delta)$ (the residue map preserves the $q$-degree).

\begin{lemma}\label{lemma:koszul}
The residue complex $\mathbb W_\bullet (\Delta'_0)$ is exact if $\Delta_0$ is an infinite face.
If $\Delta'_0$ is finite, then $\mathbb W_\bullet^q (\Delta'_0)$ has homology $H_0=\Q$ and $H_{>0} =0$ for each $q=0, \dots, \dim \Delta'_0$.
\end{lemma}

\begin{proof}
Recall that any face of $X$ is the product $\Delta'_0=\Delta \times \bar \square $, where $\Delta$ is the unimodular simplex, and $\bar \square$ is the unimodular cone (compactified by the sedendary faces). Then, 
we can write the residue complex 
as the tensor product 
(with the usual alternating sign convention) 
$\mathbb W_\bullet (\Delta'_0) = \mathbb W_\bullet (\Delta) \otimes \mathbb W_\bullet (\bar\square)$. 
The residue complex $\mathbb W_\bullet (\bar\square)$ is, in turn, the $\dim (\bar\square)$-tensor power 
of the three-term residue complex $\mathbb W_\bullet(\bar\square^1)$ 
for 1-dimensional infinite face $\bar\square^1$:
$$W_0 \<\operatorname{vertex_1}\>  \leftarrow W_0\<\operatorname{ray}\>  \oplus W_0\<\operatorname{vertex_0}\> \leftarrow W_1\< \operatorname{ray}\> ,
$$
which is acyclic. Thus, we are reduced to the case when $\Delta'_0=\Delta$ is a simplex. 

We calculate the homology of $\mathbb W_\bullet (\Delta)$ by induction on dimension of $\Delta$. The case when $\Delta$ is a point is clear: $H_0=\Q$ in degree $q=0$. For the induction step pick $\nu$, a vertex of $\Delta$, and let $\square_\nu$ be the relative cone of $\Delta$ at $\nu$. Then, we have a short exact sequence of complexes 
$$0\to \mathbb W_\bullet (\Delta\setminus \nu) \to \mathbb W_\bullet (\Delta) \to \mathbb W_\bullet (\square_\nu) \to 0,
$$
where $\mathbb W_\bullet (\Delta\setminus \nu)$ is the residue complex of the face $(\Delta\setminus \nu)\prec \Delta$, and the quotient $\mathbb W_\bullet (\square_\nu)$ is the residue complex of the cone $\square_\nu$ (without the sedentary faces). The residue complex $\mathbb W_\bullet (\square_\nu)$ is the $\dim(\square_\nu)$-tensor power of one-dimensional two-term complex
$$W_0\<\operatorname{ray}\>  \oplus W_0\<\operatorname{vertex}\> \leftarrow W_1\< \operatorname{ray}\> ,
$$
which has homology $H_0 = \Q$ in degree $q=1$ and all other vanish. Thus, $\mathbb W_\bullet (\square)$ has homology $H_0= \Q$ in degree $q=\dim (\square_\nu)=\dim (\Delta)$ and all other vanish.

Note that all maps here respect $q$-grading. Then,
the induced long exact sequence in homology 
(which becomes the short exact sequence in $H_0$'s)
does the induction step.
\end{proof} 

\begin{remark}
The proof works over $\Z$. The unimodularity of $\Delta_0'$ is irrelevant over $\Q$. 
\end{remark}

\begin{proof}[Proof of Proposition \ref{prop:SI}]
The map $\partial_0$ acts linearly with respect to the $H_{2l}(\Delta'_0)$ factors (after identifying the groups $H_{2l}(\Delta')=H_{2l}(\Delta'_0)$ within the same family). Thus, we can decompose 
$$E^0_{m,\bullet}= \Gr_{m}^{F} K^{(p)}_{\bullet+m} 
= \oplus_{l=0}^\infty \bigoplus H_{2l}(\Delta'_0) \otimes W_{\bullet+s}(\Delta)
$$ 
into direct sum of complexes $H_{2l}(\Delta'_0)\otimes \mathbb W^{m-l}_\bullet(\Delta'_0)$ indexed by $(p-2l+m)$-dimensional mobile faces $\Delta'_0$ of $X$. By Lemma \ref{lemma:koszul} each $\mathbb W^{m-l}_\bullet(\Delta'_0)$ is acyclic unless $\Delta'_0$ is a finite face, in which case it is a resolution of $\Q$. Thus, the $E^1$ term consists of the single row $E^1_{m, 0}$, which coincides with the  complex \eqref{eq:E^1}
$$ \tilde{E}^1_{m,p}=\oplus_{l=0}^{\min \{m,p\}} \bigoplus H_{2l}(\Delta_0').
$$

It only remains to check that the differential $d_1=\delta_1+\partial_1$ agrees with $\gys+i_*$.
In the proof of Proposition \ref{prop:gysin} we already identified $\delta_1$ with the Gysin map: $\gys: H_{2l}({Z}^{(k)}) \to H_{2l-2}(Z^{(k+1)})$. Now we analyze what remains of the horizontal map $\partial: K^{(p)}_{\bullet} \to K^{(p)}_{\bullet-1}$ on the $E^1$ level.

Let $\Delta'$ be a finite $(p-2l-m)$-dimensional mobile face of $X$,
and let $\nu\in \Delta'$ be its vertex. Let $\bar \beta \in H_{2l}(\Delta') $ be an element of $E^1_{m,0}$. 
We want to compute the image $\partial_1 \bar\beta$ in $H_{2l}(\Delta''')$ for $\Delta''' = (\Delta' \setminus \nu)$. 

First, we represent $\bar\beta$ by an element 
$$\beta\otimes 1 \in H_{2l}(\Delta') \otimes W_0 (\Delta) \subset E^0_{m,0}= \Gr^F_m K^{(p)}_m,
$$
 where $\Delta$ is an $(m-l)$-dimensional face of $\Delta'$ containing the vertex $\nu$. Next, a lift of $\beta\otimes 1$ to $A^{(p)}_l(\Delta) \subset K^{(p)}_{m-l, l}$ can be represented by $\beta\otimes w \in H_{2l}(\Delta') \otimes W_{p-l} (\Delta')$, where $w$ is a relative volume polyvector for the pair $\Delta'\succ\Delta$. Moreover, we can take this volume polyvector to be in the form $w=\iota (u)$ where $u\in W_{p-l} (\Delta''')$.
Then, $\shad_{\Delta'''\succ\Delta''}(u) = 1 \in W_0 (\Delta'')$, where $\Delta''=(\Delta\setminus \nu)$.  

On the other hand, $\dd (\beta \otimes w) = \beta\otimes w \sim i_*(\beta) \otimes u$ in $A^{(p)}_l(\Delta'')$. In $E^0_{m-1,0}$ 
({\it i.e.}, after applying the residue map) this becomes $i_*(\beta) \otimes 1 \in  H_{2l}(\Delta''') \otimes W_0 (\Delta'')$. 
Thus, $\partial_1$ is the pushforward map $i_*$.
\end{proof}

\end{document}